\documentclass[10pt]{amsart}
\usepackage{
  amsmath,
  amssymb,
  amsthm,
  geometry,
  paralist,
  thmtools,
  tikz,
  tikz-cd,
  verbatim,
  relsize, 
  bbm, 
  dsfont,
  mathtools
}

\usepackage{multicol} 

\usepackage[T1]{fontenc}
\usepackage[final]{microtype}

\usepackage[style=numeric]{biblatex}
\usepackage[unicode, psdextra]{hyperref}

\usepackage{xurl}

\hypersetup{
    colorlinks=true,
    linkcolor=blue,
    filecolor=magenta,      
    urlcolor=cyan,
    citecolor = blue,
	breaklinks=true
}

\usepackage{cleveref} 

\usepackage{adjustbox}

\tikzcdset{
	arrows = crossing over
}

\usepackage{subfigure}

\newtheorem{thmx}{Theorem}

\usepackage{dynkin-diagrams}

\usetikzlibrary{arrows.meta} 

\usepackage[textsize=scriptsize]{todonotes}
\presetkeys{todonotes}{inline}{}

\makeatletter
\providecommand\@dotsep{5}
\makeatother

\newcommand{\bbC}{\mathbb C}

\newcommand{\bbR}{\mathbb R}

\newcommand{\bbZ}{\mathbb Z}

\newcommand{\cA}{\mathcal A}

\newcommand{\cC}{\mathcal C}

\newcommand{\cF}{\mathcal F}
\newcommand{\cG}{\mathcal G}

\newcommand{\cU}{\mathcal U}

\newcommand{\fri}{\mathfrak{i}}
\newcommand{\frc}{\mathfrak{c}}
\newcommand{\frk}{\mathfrak{k}}

\newcommand{\ufld}{\text{ufld}}

\newcommand{\catname}[1]{{\mathsf{#1}}}

\newcommand{\TLJ}{\catname{TLJ}}
\renewcommand{\mod}{\catname{mod}}

\renewcommand{\vec}{\catname{vec}}
\newcommand{\rep}{\catname{rep}}
\DeclareMathOperator{\cEnd}{\catname{End}}

\newcommand{\<}{\langle}
\renewcommand{\>}{\rangle}

\DeclareMathOperator{\id}{id}
\DeclareMathOperator{\Z}{\mathbb Z}

\newcommand{\1}{\mathds{1}}

\DeclareMathOperator{\ra}{\rightarrow}
\newcommand{\xra}{\xrightarrow}

\DeclareMathOperator{\Hom}{Hom}

\DeclareMathOperator{\img}{im}
\DeclareMathOperator{\coker}{coker}

\DeclareMathOperator{\Irr}{Irr}
\DeclareMathOperator{\Path}{Path}

\declaretheorem[numberwithin=section]{theorem}
\declaretheorem[sibling=theorem]{lemma}

\declaretheorem[sibling=theorem]{proposition}
\declaretheorem[sibling=theorem, style=remark]{remark}
\declaretheorem[sibling=theorem, style=definition]{definition}
\declaretheorem[sibling=theorem, style=definition]{example}

\newtheorem*{theorem*}{Theorem}

\crefname{lemma}{Lemma}{Lemma}
  \crefname{corollary}{Corollary}{Corollary}
  \crefname{theorem}{Theorem}{Theorem}
  \crefname{definition}{Definition}{Definition}
   \crefname{proposition}{Proposition}{Proposition}
 \crefname{section}{Section}{Section} 
   \crefname{construction}{Construction}{Construction}
   \crefname{generalization}{Generalization}{Generalization}
  \crefname{construction}{Construction}{Construction}
  \crefname{notation}{Notation}{Notation}
   \crefname{example}{Example}{Example}
  \crefname{remark}{Remark}{Remark}
  \crefname{fact}{Fact}{Fact}
  \crefname{conjecture}{Conjecture}{Conjecture}
  \crefname{motivation}{Motivation}{Motivation}  
  \crefname{figure}{Figure}{Figure}

\newcommand{\nPi}[1][n]{{ \overset{\scriptscriptstyle #1}{\Pi}}}

\newcommand{\bigboxtimes} {
    {\mathrel{\raisebox{-.6ex}
    	{$\mathlarger{\mathlarger{\mathlarger{\boxtimes}}} $}
    	}
    }
}

\newcommand{\bigboxtimesp}[2] {
    \underset{#2}{\overset{#1}{
    	{\mathrel{
    		\raisebox{-.6ex} 
    		{$\mathlarger{\mathlarger{\mathlarger{\mathlarger{\boxtimes}}}} $}
    		}
    	}
    }
    }
}

\newcommand{\tri}{%
\mathrel{
\begin{tikz}[line cap=round, line join=round]%
  \draw[-] (0, 1ex) -- (0,2ex)
(0, 1ex) -- (1ex,0ex)
(0, 1ex) -- (-1ex,0ex);
\end{tikz}%
}
}

\newcommand{\itri}{%
\begin{tikz}[line cap=round, line join=round]%
  \draw[-] (0, 1ex) -- (0,0ex)
(0, 1ex) -- (1ex,2ex)
(0, 1ex) -- (-1ex,2ex);
\end{tikz}%
}

\newcommand{\xtwoheadrightarrow}[2][]{%
  \xrightarrow[#1]{#2}\mathrel{\mkern-14mu}\rightarrow
}

\renewcommand{\comment}[1]{}

\usepackage{quiver} 

\calclayout

\begin{document}
\title[Coxeter quivers in fusion categories]{Coxeter quiver representations in fusion categories and Gabriel's theorem}
\author{Edmund Heng}
\address{Institut des Hautes Etudes Scientifiques (IHES). Le Bois-Marie, 35, route de Chartres, 91440 Bures-sur-Yvette (France)}
\email{heng@ihes.fr}



\begin{abstract}
We introduce a notion of representation for a class of generalised quivers known as \emph{Coxeter quivers}.
These representations are built using fusion categories associated to $U_q(\mathfrak{s}\mathfrak{l}_2)$ at roots of unity and we show that many of the classical results on representations of quivers can be generalised to this setting.
Namely, we prove a generalised Gabriel's theorem for Coxeter quivers that encompasses all \emph{Coxeter--Dynkin diagrams} -- including the non-crystallographic types $H$ and $I$.
Moreover, a similar relation between reflection functors and Coxeter theory is used to show that the indecomposable representations correspond bijectively to the positive roots of Coxeter root systems over fusion rings.
\end{abstract}

\maketitle

\section*{Introduction}
Recall the following well-celebrated result of Gabriel on the classification of finite-type quivers:
\begin{theorem*}[Gabriel's theorem \cite{Gab72}]
A quiver $\check{Q}$ has finitely-many isomorphism classes of indecomposable representations (finite-type) if and only if its underlying graph $\overline{\check{Q}}$ is a (finite, disjoint) union of ADE Dynkin diagrams: $A_n, D_n, E_6, E_7, E_8$.
Moreover, in these cases the isomorphism classes of indecomposable representations correspond bijectively to the positive roots of the root system associated to its underlying graph $\overline{\check{Q}}$ by taking the dimension vector.
\end{theorem*}
\noindent By the works of Dlab--Ringel \cite{Dlab_Ring75, DR_76}, it is known that the result above can be extended to include all Dynkin diagrams (adding in $B_n,C_n, F_4, G_2$) if one considers representations of a generalised notion of quivers known as valued quivers (or $K$-species).

The generalisation to valued quivers is quite natural as the $ADE$ diagrams are part of the Dynkin diagrams classifying the semisimple Lie algebras, which all have well-defined root systems.
However, the $ADE$ diagrams are also part of a similar (yet slightly different) classification known as the \emph{Coxeter--Dynkin diagrams}, which serve to classify the finite Coxeter groups \cite{Cox35}.
More importantly, root systems for Coxeter groups can be similarly developed; see for example \cite{Deodhar82, Dyer09}.
It is therefore natural to search for a generalised Gabriel's theorem that is related to Coxeter theory -- such is the aim of this paper.
Namely, we shall introduce a notion of representation for a (different) class of generalised quivers known as \emph{Coxeter quivers}, which will extend Gabriel's theorem to include all Coxeter--Dynkin diagrams -- including the non-crystallographic types $H_3, H_4$ and $I_2(n)$.

A \emph{Coxeter quiver} (see \cref{defn: Coxeter quiver}) is a quiver $Q$ whose arrows are labelled by natural numbers in $\{3,4,5,...\} \subset \mathbb{N}$, where a classical quiver is (by definition) a Coxeter quiver with all arrows labelled by `3'.
Whilst representations of quivers are built using the category of vector spaces, representations of Coxeter quivers $Q$ are built using different \emph{fusion categories} $\cC(Q)$ depending on $Q$.
These fusion categories $\cC(Q)$ are constructed from (tensor products of) fusion categories associated to $U_q(\mathfrak{s}\mathfrak{l}_2)$ at roots of unity\footnote{We will actually use the equivalent diagrammatic category known as the Temperley-Lieb-Jones categories $\TLJ_n$.} (see \cref{subsec: Coxeter quiver rep}); when $Q$ is just a classical quiver, $\cC(Q)$ will be (equivalent to) the category of vector spaces.

A representation $V$ of a Coxeter quiver $Q$ consists of the following datum:
to each vertex $i \in Q_0$, we assign an object $V_i$ in the fusion category $\cC(Q)$; and to each arrow $\alpha$ from $i \ra j$ with label $n$, we assign a morphism $V_{\alpha} : A_n \otimes V_i \ra V_j$ with $A_n$ depending on the label $n$ (see \cref{defn: coxeter quiver representation} for the precise definition).
In the particular case where $Q$ is just a classical quiver (all labels are `3'), $A_n = A_3$ will be the one-dimensional vector space, so that the definition coincides with the classical quiver representation.
These representations, together with their morphisms, form an abelian category in a similar fashion, and we say that a Coxeter quiver is of \emph{finite-type} if its category of representations has finitely many indecomposables (up to isomorphism).
The relevance to the first part of Gabriel's theorem is as follows:
\begin{thmx}[= \cref{thm: gabriel classification}] \label{thm: intro Gabriel}
A Coxeter quiver  $Q$ is of finite-type if and only if its underlying graph (forgetting orientations but keeping the labels) is a finite (disjoint) union of Coxeter--Dynkin diagrams: $A_n, B_n=C_n, D_n, E_6,E_7,E_8, F_4, G_2$, $H_3, H_4, I_2(n)$.
\end{thmx}

While representations of classical quivers are related to root system over $\Z$ defined on lattices $L_{\check{Q}}:= \bigoplus_{i \in {\check{Q}}_0} \Z$, the representations of Coxeter quivers are related to generalised root systems defined over (commutative) \emph{fusion rings} $K_0(\cC(Q))$ of $\cC(Q)$ (see \cref{defn: fusion ring} for definition of fusion rings).
More precisely, given a Coxeter quiver $Q$, the associated root system is defined on the ``lattice'' $L_Q := \bigoplus_{i \in Q_0} K_0(\cC(Q))$ equipped with a bilinear form that is linear in $K_0(\cC(Q))$ and takes value in $K_0(\cC(Q))$ as well.
This type of root systems were studied by Dyer in \cite{Dyer09} and many of the properties satisfied by the usual root systems can be generalised to this setting.
In particular, the bilinear form induces a (faithful) action of the corresponding Coxeter group and the notion of positive roots (and their extended version) can be defined in a similar fashion (see \cref{subsec: root system in fusion ring}).

Analogous to the classical case, the dimension vector of a Coxeter quiver representation (see \cref{defn: dim vector}) will take value in $L_Q$.
Using reflection functors for Coxeter quivers (similar to those in \cite{BGP_73, DR_76}) and their interaction with the generalised root systems described above, we obtain the second part of Gabriel's theorem for Coxeter quivers.
\begin{thmx}[= \cref{thm: gabriel root}] \label{thm: intro gabriel root}
Let $Q$ be a Coxeter quiver whose underlying graph is a Coxeter--Dynkin diagram. 
The isomorphism classes of indecomposable representations of $Q$ correspond bijectively to the extended positive roots of the free $K_0(\cC(Q))$-module $L_{Q}$ by taking the dimension vector.
Moreover, every indecomposable representation $V \in \rep(Q)$ is of the form
\[
V \cong V' \otimes A
\]
for some $V' \in \rep(Q)$ that corresponds to a positive root and some simple object $A$ in $\cC(Q)$.
\end{thmx}

The approach of studying non-simply-laced Dynkin graphs and quivers by folding their simply-laced cousins via automorphisms is well-known; see for example \cite{Stein68, Lusz_book_intro_quant, Tanisaki80}.
However, certain non-crystallographic graphs require foldings that do not arise from automorphisms.
One famous example is Lusztig's folding of the $E_8$ graph onto the $H_4$ graph \cite{Lusz83}, with other similar examples studied in different contexts \cite{Scherbak88, Moody_Patera93, Zuber94, Crisp99, Dyer09}.
Although the theory of Coxeter quiver representations (a priori) may not seemed related to the classical theory of quiver representations, they turn out to be related through these types of foldings.
\begin{thmx}[= \cref{thm: rep Q and checkQ equiv}] 
Let $Q$ be a Coxeter quiver and let $\check{Q}$ be its associated unfolded classical quiver (see \cref{subsec: unfolding}).
The two categories $\rep(Q)$ and $\rep(\check{Q})$ are equivalent as abelian categories.
\end{thmx}
It is important to note that the equivalence is only an equivalence of abelian categories; $\rep(Q)$ is also a (right) \emph{module category} (see \cref{defn: module category}) over the fusion category $\cC(Q)$.
In other words, $\rep(Q)$ and the equivalence above provide an explicit way of realising a $\cC(Q)$-module category structure on $\rep(\check{Q})$, which can be seen as a folding structure on $\check{Q}$.\footnote{In the sense of \cite{DufTumarkin22}; see \cref{subsec: cluster and folding} for more details.}
Let us also mention here that with the appropriate notion of \emph{path algebra} $\Path(Q)$ for $Q$ (see \cref{defn: path alg}), we also have an equivalence between the category of modules $\Path(Q)$-$\mod$ and $\rep(Q)$; see \cref{sec: path algebra} for more details.

Now that we have revealed the connection to classical quiver representations, the observant reader may realise that the first part of Gabriel's theorem for Coxeter quivers (\Cref{thm: intro Gabriel}) is just a simple consequence of this theorem, and that is indeed the case -- one just has to show that $Q$ is finite type if and only if $\check{Q}$ is finite type.
This will only involve a rather simple combinatorial check (plus some Coxeter theory).

\subsection*{Outline of the paper}
Here is a brief outline of the paper.
\cref{sec: TLJ} contains the background materials, where we define fusion categories and introduce the fusion categories that will be used to construct representations of Coxeter quivers.
\cref{sec: Coxeter quiver and rep} contains the definition of Coxeter quivers and their representations.

\cref{sec: unfolding and gabriel classification} and \cref{sec: reflection and roots} contain the main results. 
In \cref{sec: unfolding and gabriel classification}, we define the unfolding of Coxeter quivers $Q$ to obtain their associated classical quiver $\check{Q}$ and moreover, relate their categories of representations.
This will allow us to prove the first part of Gabriel's theorem for Coxeter quivers (\Cref{thm: intro Gabriel} = \cref{thm: gabriel classification}).
\cref{sec: reflection and roots} is where we define root systems over fusion rings and reflection functors for Coxeter quivers, which we use to prove the second part of Gabriel's theorem for Coxeter quivers (\Cref{thm: intro gabriel root} = \cref{thm: gabriel root}).

\cref{sec: path algebra} is dedicated to the complementary result that the category of modules over the path algebra of a Coxeter quiver is equivalent to the category of Coxeter quiver representations.
Finally \cref{sec: further remarks} contains some remarks on related works and possible further directions.

\subsection*{Notation and Convention}
Throughout this paper, we will use hook arrows $\hookrightarrow$ and two-headed arrows $\twoheadrightarrow$ to represent monomorphisms and epimorphisms respectively.

Given an abelian category $\cC$, the image of a morphism $f: X \ra Y$ is by definition the kernel of the cokernel of $f$, captured by the diagram
\[
\begin{tikzcd}
& \img(f) := \ker(\coker(f)) \ar[d, hookrightarrow,, "\fri(f):= \frk(\frc(f))"] & \\
X \ar[r,"f"] \ar[ru, dashed, twoheadrightarrow, "\exists ! \overline{f}"] & Y \ar[r, twoheadrightarrow, "\frc(f)"] & \coker(f).
\end{tikzcd}
\]
In particular, the unique epi-mono factorisation of a morphism $X \xra{f} Y$ through its image shall be denoted by $X \xtwoheadrightarrow{\overline{f}} \img(f) \xhookrightarrow{\fri(f)} Y$.

Given a direct sum $\bigoplus_{i \in I} Y_i$ for some index set $I$ with maps $f_i: X \ra Y_i$ for each $i \in Q_0$, we shall use $[f_i]_{i \in I}$ to denote the map from $X \ra \bigoplus_{i \in I} V_i$ induced by the $f_i$'s into each summand, thinking of it as a column vector of maps once we fix some ordering on the index set $I$.
Similarly, given maps $g_i: Y_i \ra Z$, we shall use $[g_i]_{i \in I}^T$ to denote the map from $\bigoplus_{i \in I} Y_i \ra Z$ induced by the $g_i$ from each summand, thinking of it as a row vector of maps.
In particular, the map 
\[
X \xra{[f_i]_{i \in I}} \bigoplus_{i \in I} Y_i \xra{[g_i]_{i \in I}^T} Z
\] 
is just $\sum_{i \in I} g_i \circ f_i$, which is the same as the matrix multiplication $[g_i]_{i \in I}^T [f_i]_{i \in I}$ once an ordering of $I$ is fixed.
On the other hand, if we were given a map $f: \bigoplus_{i \in I} X_i \ra \bigoplus_{i \in I} Y_i$, we shall use $(f)_{i,j}$ to denote the map between the summands $(f)_{i,j}: X_i \ra Y_j$.
We will also use $A \overset{\oplus}{\subseteq} B$ to mean $A$ appears as a direct summand of $B$ (up to isomorphism).

In this paper, two very similar sounding notions will appear: a \emph{module category} over a fusion (tensor) category $\cC$ (\cref{defn: module category}) and a \emph{category of modules} over an algebra in a fusion category (\cref{defn: modules over algebra}). 
We warn the reader that these are two different (although related) notions and not to confuse the two of them.

\subsection*{Acknowledgements}
I'd like to first thank my PhD supervisor, Tony Licata, for his guidance and questions which led to the making of this work.
I'd also like to thank Drew D. Duffield and Greg Kuperberg for the interesting discussions during the preparation of this work.

\section{Preliminaries: fusion categories and \texorpdfstring{$\TLJ_n$}{TLJ}} \label{sec: TLJ}
\subsection{Fusion tensor categories and their module categories}
We shall quickly review the notion of a fusion category and some of its (desired) properties. 
We refer the reader to \cite{etingof_nikshych_ostrik_2005, EGNO15} for more details on this subject.

\begin{definition}\label{defn: tensor fusion category}
A (strict) \emph{tensor category} $\cC$ is a $\mathbb{K}$-linear (hom spaces are $\mathbb{K}$-vector spaces), abelian and (strict) monoidal category with a bi-exact monoidal structure $-\otimes -$, satisfying the following properties:
\begin{itemize}
\item the monoidal unit $\1$ is simple; and
\item rigidity: every object has left and right duals.
\end{itemize}
We say that $\cC$ is a \emph{fusion category} if $\cC$ is semi-simple and its set of isomorphism classes of simple objects is finite.
\end{definition}

\begin{example}
At this point the reader should check that the simplest possible example of a fusion category is the category of finite dimensional $\mathbb{K}$-vector spaces.
\end{example}

Other fusion categories of interests will be defined in \cref{subsection: TLJ}.

\begin{definition}\label{defn: fusion ring}
Let $\cA$ be an abelian category.
Its \emph{Grothendieck group} $K_0(\cA)$ is the abelian group freely generated by isomorphism classes $[A]$ of objects $A \in \cA$ modulo the relation
\[
[B] = [A] + [C] \iff 0 \ra A \ra B \ra C \ra 0 \text{ is exact}.
\]
If $\cA = \cC$ is moreover a fusion category, $K_0(\cC)$ can be equipped with a unital $\mathbb{N}$-ring structure.
Namely, it has a basis given by the (finite) isomorphism classes of simples $\{ [S_i] \}_{i=0}^k$ with $S_0 := \1$ denoting the monoidal unit of $\cC$, and multiplication is defined by the tensor product $[A]\cdot[B] := [A\otimes B]$, which satisfies
\[
[A] \cdot [B] = \sum_{i=0}^k r^{AB}_i [S_i], \quad r^{AB}_i \in \mathbb{N}.
\]
This equip $K_0(\cC)$ with an unital, associative ring structure with unit $[S_0] = [\1]$ and we call $K_0(\cC)$ the \emph{fusion ring} of $\cC$.
\end{definition}

\begin{definition} \label{defn: module category}
A right \emph{module category} over a fusion category $\cC$ is an additive category $\cA$ together with an additive monoidal functor $\Psi: \cC^{\otimes^\text{op}} \ra \cEnd(\cA)$.
We will usually denote the endofunctor $\Psi(A)$ assigned to $A \in \cC$ by $- \otimes A$.
\end{definition}
Note that when $\cA$ is moreover abelian, $\Psi$ will automatically map into the category of \emph{exact} endofunctors, since every object in a fusion category $\cC$ has both left and right duals.
As such, the Grothendieck group $K_0(\cA)$ of $\cA$ is also naturally a module over the fusion ring $K_0(\cC)$.
\begin{example}
\begin{enumerate}[(i)]
\item Any fusion category is always a right module category over itself, sending each object $A$ to the functor $- \otimes A$.
\item One can check that every additive category $\cA$ is naturally a module category over the category of finite dimensional vector spaces $\vec_{\mathbb{K}}$, sending each (non-zero) vector space $V$ to the endofunctor of $\cA$ which takes $A\in \cA$ to 
$
V(A) := \bigoplus_{i=1}^{\dim (V)} A,
$
and every morphism $A \xra{f} A'$ in $\cA$ to 
$
V(A) \xra{\bigoplus_{i=1}^{\dim V} f} V(A').
$
\end{enumerate} 
\end{example}

\subsection{Temperley-Lieb-Jones category at root of unity} \label{subsection: TLJ}
The fusion categories of interest in this paper will be built from (tensor products of) strict fusion categories known as the \emph{Temperley-Lieb-Jones category evaluated at} $q = e^{i\pi/n}$, which we denote as $\TLJ_n$.
This is a diagrammatic category, constructed as the additive completion of the category of Jones-Wenzl projectors that is semi-simplified by killing the negligible $(n-1)$th Jones-Wenzl projector.
In what follows, we will provide the minimal description of what is needed for the purpose of our paper, where we follow the description in \cite{chen2014temperleylieb} closely.
For more details, we refer the reader to \cite{wang_2010, turaev_2016}.
\begin{remark}
The representation-theoretic-inclined reader could instead use the semi-simplified category $\overline{\rep}(U_q(\mathfrak{s}\mathfrak{l}_2))$ of $U_q(\mathfrak{s}\mathfrak{l}_2)$-modules with $q = e^{i\frac{\pi}{n}}$, which is equivalent to $\TLJ_n$ as (braided) fusion categories; see \cite[Chapter XI, Section 6]{turaev_2016} and \cite[Section 5.5]{SnyTing09}.
\end{remark}

The Temperley-Lieb-Jones category $\TLJ_n$ at $q = e^{i\pi/n}$ \cite[Definition 5.4.1]{chen2014temperleylieb} is a strict fusion category that is $\bbC$-linear and is generated additively and monoidally by the following pairwise non-isomorphic $n-1$ simple objects:
\[
\Pi_0, \Pi_1, ..., \Pi_{n-2},
\]
with monoidal unit $\1$ given by $\Pi_0$.
Each simple object in $\TLJ_n$ is isomorphic to some (uniquely determined) $\Pi_a$, and we fix $\Pi_a$ to be the representative of each of the corresponding isomorphism class of simple objects.
In particular, the set of representatives of isomorphism classes of simple objects $\Irr(\TLJ_n)$ is given by
\[
\Irr(\TLJ_n) = \{ \Pi_0, \Pi_1, ..., \Pi_{n-2} \}.
\]

The simple objects $\Pi_a$ are self-dual (hence all objects are self dual).
In particular, we will use a cap (resp. cup) to denote the counit (resp. unit) of the self-dual structure on each $\Pi_a$ for $0 \leq a \leq n-2$:
\begin{equation} \label{dual counit}
\begin{tikzpicture}[scale = 0.4]
	\node at (0,-0.5) {$\Pi_a$};
	\node at (2,-0.5) {$\Pi_a$};
	\node at (1, 2) {$\Pi_0$};
  \draw (0, 0) arc[start angle=180,end angle=0,radius=1];
\end{tikzpicture}, \quad
\begin{tikzpicture}[scale = 0.4]
	\node at (0,2) {$\Pi_a$};
	\node at (2,2) {$\Pi_a$};
	\node at (1, -0.5) {$\Pi_0$};
  \draw (0, 1.5) arc[start angle=-180,end angle=0,radius=1];
\end{tikzpicture},
\end{equation}
satisfying the usual ``isotopy of strings'' relation (straight lines denote identity maps):
\[
\begin{tikzpicture}[scale = 0.4]
	\node at (4, 2.5) {$\Pi_a$};
  \draw (0, .5) arc[start angle=180,end angle=0,radius=1];
  \draw (4,.5) -- (4,2);
	\node at (0, 0) {$\Pi_a$};
	\node at (2, 0) {$\Pi_a$};
	\node at (4,0) {$\Pi_a$};
  \draw (0,-2) -- (0,-.5);
  \draw (2, -.5) arc[start angle=-180,end angle=0,radius=1];
	\node at (0, -2.5) {$\Pi_a$};
	\node at (6, 0) {$=$};
	\node at (8, 2.5) {$\Pi_a$};
  \draw (8,-2) -- (8,2);
  	\node at (8.5, 0) {id};
	\node at (8, -2.5) {$\Pi_a$};
	\node at (10, 0) {$=$};
	\node at (12, 2.5) {$\Pi_a$};
  \draw (12,.5) -- (12,2);
  \draw (14, .5) arc[start angle=180,end angle=0,radius=1];
	\node at (12, 0) {$\Pi_a$};
	\node at (14, 0) {$\Pi_a$};
	\node at (16,0) {$\Pi_a$};
  \draw (12, -.5) arc[start angle=-180,end angle=0,radius=1];
  \draw (16,-2) -- (16,-.5);
	\node at (16, -2.5) {$\Pi_a$};
\end{tikzpicture}.
\]
(Note that $\Pi_0$ were omitted since $\TLJ_n$ is a strict monoidal category, so that $\Pi_0 \otimes \Pi_a = \Pi_a = \Pi_a \otimes \Pi_0$.)

The semi-simple decompositions of the objects in $\TLJ_n$ is described by the following \emph{fusion rule}:
\begin{equation} \label{eq: fusion rule}
\Pi_a \otimes \Pi_b \cong 
	\begin{cases}
	\Pi_{|a-b|} \oplus \Pi_{|a-b| + 2} \oplus \cdots \oplus \Pi_{a+b}, &a+b \leq n-2; \\
	 \Pi_{|a-b|} \oplus \Pi_{|a-b| + 2} \oplus \cdots \oplus \Pi_{2n-(a+b)-4}, &a+b > n-2.
	\end{cases}
\end{equation}
Notice that no simple appears more than once in the semi-simple decomposition above.

\begin{remark}
Equation \eqref{eq: fusion rule} is symmetric along $a$ and $b$, so we have that $\Pi_a \otimes \Pi_b \cong \Pi_b \otimes \Pi_a$.
Nonetheless, the category $\TLJ_n$ is not symmetric, but \emph{braided}; see \cite[Chapter XII, Section 6]{turaev_2016}.
\end{remark}

The cases where $a=1,n-2$ and $n-3$ will be of utmost importance to us, which we shall explicitly record here for convenience:
\begin{itemize}
\item $a=1$:
\begin{equation} \label{eq: 1 b fusion rule}
\Pi_1 \otimes \Pi_b \cong \Pi_b \otimes \Pi_1 \cong
	\begin{cases}
	\Pi_1, &b=0 \\
	\Pi_{n-3}, &b=n-2\\
	\Pi_{b-1} \oplus \Pi_{b+1}, &\text{ otherwise}.
	\end{cases}
\end{equation}
\item $a = n-2$:
\begin{equation} \label{eq: z/2 fusion rule}
\Pi_{n-2} \otimes \Pi_b \cong \Pi_{n-2-b}.
\end{equation}
\item $a=n-3$: Together \eqref{eq: 1 b fusion rule} and \eqref{eq: z/2 fusion rule} imply that
\begin{equation} \label{eqn: n-3 b fusion rule}
\Pi_{n-3} \otimes \Pi_b \cong \Pi_{n-2} \otimes \Pi_1 \otimes \Pi_b \cong 
	\begin{cases}
	\Pi_{n-3}, &b=0 \\
	\Pi_{1}, &b=n-2\\
	\Pi_{n-2-b-1} \oplus \Pi_{n-2-b+1}, &\text{ otherwise}.
	\end{cases}
\end{equation}
\end{itemize}

Using the fusion rules in \eqref{eq: fusion rule}, one sees that if $a$ and $b$ are both even, then the semi-simple decomposition contains only evenly labelled objects.
We denote the full fusion subcategory of $\TLJ_n$ generated the evenly labelled object by $\TLJ_n^{even}$, which has set of simple representatives given by 
\[
\Irr(\TLJ^{even}_n) = \{\Pi_a \mid a \text{ is even}\}.
\]
Note that when $n = 3$, the fusion subcategory $\TLJ_3^{even}$ is equivalent to the category of finite dimensional vector spaces over $\bbC$.

It follows from semi-simplicity and \eqref{eq: fusion rule} that we have an explicit description of the following morphism spaces:
\begin{equation} \label{eqn: hom space trivalent}
\Hom_{\TLJ_n}(\Pi_a \otimes \Pi_b, \Pi_c) = 
\begin{cases}
\bbC\{\tri\}, &\text{ if } \Pi_c \overset{\oplus}{\subseteq} \Pi_a \otimes \Pi_b; \\
0,  &\text{ otherwise},
\end{cases}
\end{equation}
where $\tri$ denotes a choice of basis element for the one-dimensional morphism space known as the \emph{$q$-admissible map} \cite[Section 3.1]{chen2014temperleylieb}:
\begin{equation} \label{admissible triple}
\begin{tikzpicture}[scale = 0.4] 
  \node at (0,1.5) {$\Pi_c$};
  \node at (1,-1.5) {$\Pi_b$};
  \node at (-1,-1.5) {$\Pi_a$};
  \draw (0, 0) to (0,1);
  \draw (0, 0) to (1,-1);
  \draw (0, 0) to (-1,-1);
\end{tikzpicture}.
\end{equation}
When $c=0$, it follows from \eqref{eq: fusion rule} that the only case when the morphism space is non-zero is $a=b$ and the $q$-admissible map $\tri$ in this case coincides with the cap map $\cap$ in \eqref{dual counit}.
On the other hand when $a=0$ (resp. $b=0$), it must be that $c=a$ (resp. $c=b$) and the $q$-admissible map is the identity map on $\Pi_c$.
The (non-unique) isomorphism in the fusion rule \eqref{eq: fusion rule} can also be given by summing up the corresponding $q$-admissible map into each simple summand.

The $q$-admissible maps represented by the trivalent graphs are ``rotational invariant'', namely under the dualisations
\[
\Hom_{\TLJ_n}(\Pi_a \otimes \Pi_b, \Pi_c) \xra{\cong} \Hom_{\TLJ_n}(\Pi_a,  \Pi_c \otimes \Pi_b) \xra{\cong} \Hom_{\TLJ_n}(\Pi_c \otimes \Pi_a, \Pi_b),
\]
the map obtained agrees with the corresponding $q$-admissible map:
\[
\begin{tikzpicture}[scale = 0.4] 
  \node at (0,1.5) {$\Pi_b$};
  \node at (1,-1.5) {$\Pi_a$};
  \node at (-1,-1.5) {$\Pi_c$};
  \draw (0, 0) to (0,1);
  \draw (0, 0) to (1,-1);
  \draw (0, 0) to (-1,-1);
  \node at (2, 0) {$=$};
	\node at (10, 5) {$\Pi_b$};
  \draw (4, 3) arc[start angle=180,end angle=0,radius=1.5];
  \draw (10,3) -- (10,4.5);
	\node at (4, 2.5) {$\Pi_c$};
	\node at (7, 2.5) {$\Pi_c$};
	\node at (10, 2.5) {$\Pi_b$};
  \draw (4,.5) -- (4,2);
  \draw (7, 1.2) to (7,2);
  \draw (7, 1.2) to (8,0.5);
  \draw (7, 1.2) to (6,0.5);
  \draw (10,.5) -- (10,2);
	\node at (4, 0) {$\Pi_c$};
	\node at (6, 0) {$\Pi_a$};
	\node at (8, 0) {$\Pi_b$};
	\node at (10,0) {$\Pi_b$};
  \draw (4,-2) -- (4,-.5);
  \draw (6,-2) -- (6,-.5);
  \draw (8, -.5) arc[start angle=-180,end angle=0,radius=1];
	\node at (4, -2.5) {$\Pi_c$};
	\node at (6, -2.5) {$\Pi_a$};
\end{tikzpicture};
\]
the same is true for the other rotation.

We shall use the inverted trivalent graphs:
\[
\begin{tikzpicture}[scale = 0.4] 
  \node at (0,-1.5) {$\Pi_a$};
  \node at (1,1.5) {$\Pi_b$};
  \node at (-1,1.5) {$\Pi_c$};
  \draw (0, 0) to (0,-1);
  \draw (0, 0) to (1,1);
  \draw (0, 0) to (-1,1);
  \node at (2, 0) {$:=$};
	\node at (5, 2.5) {$\Pi_c$};
	\node at (8, 2.5) {$\Pi_b$};
  \draw (5, 1.2) to (5,2);
  \draw (5, 1.2) to (6,0.5);
  \draw (5, 1.2) to (4,0.5);
  \draw (8,.5) -- (8,2);
	\node at (4, 0) {$\Pi_a$};
	\node at (6, 0) {$\Pi_b$};
	\node at (8,0) {$\Pi_b$};
  \draw (4,-2) -- (4,-.5);
  \draw (6, -.5) arc[start angle=-180,end angle=0,radius=1];
	\node at (4, -2.5) {$\Pi_a$};
\end{tikzpicture}
\qquad,
\begin{tikzpicture}[scale = 0.4] 
  \node at (0,-1.5) {$\Pi_b$};
  \node at (1,1.5) {$\Pi_c$};
  \node at (-1,1.5) {$\Pi_a$};
  \draw (0, 0) to (0,-1);
  \draw (0, 0) to (1,1);
  \draw (0, 0) to (-1,1);
  \node at (2, 0) {$:=$};
	\node at (4, 2.5) {$\Pi_a$};
	\node at (7, 2.5) {$\Pi_c$};
  \draw (7, 1.2) to (7,2);
  \draw (7, 1.2) to (8,0.5);
  \draw (7, 1.2) to (6,0.5);
  \draw (4,.5) -- (4,2);
	\node at (4, 0) {$\Pi_a$};
	\node at (6, 0) {$\Pi_a$};
	\node at (8,0) {$\Pi_b$};
  \draw (8,-2) -- (8,-.5);
  \draw (4, -.5) arc[start angle=-180,end angle=0,radius=1];
	\node at (8, -2.5) {$\Pi_b$};
\end{tikzpicture}
\]
to denote the two dualisations of the $q$-admissible map in $\Hom_{\TLJ_n}(\Pi_a \otimes \Pi_b,  \Pi_c)$ to the ($q$-admissible) maps in $\Hom_{\TLJ_n}(\Pi_a,  \Pi_c \otimes \Pi_b)$ and $\Hom_{\TLJ_n}(\Pi_b,  \Pi_a \otimes \Pi_c)$ respectively.
Once again it follows from \eqref{eq: fusion rule} and the semi-simple property of $\TLJ_n$ that 
\begin{equation} \label{eqn: hom space inverted trivalent}
\Hom_{\TLJ_n}(\Pi_a, \Pi_c \otimes \Pi_b) = 
\begin{cases}
\bbC\{\itri\}, &\text{ if } \Pi_a \overset{\oplus}{\subseteq} \Pi_c \otimes \Pi_b; \\
0,  &\text{ otherwise}.
\end{cases}
\end{equation}
It also follows from the rotational invariant property that we have the following relation between the $q$-admissible maps:
\begin{equation} \label{frobenius rel}
\begin{tikzpicture}[scale = 0.4] 
	\node at (-2, 2.5) {$\Pi_b$};
	\node at (1, 2.5) {$\Pi_d$};
  \draw (-2, 1.2) to (-2,2);
  \draw (-2, 1.2) to (-1,0.5);
  \draw (-2, 1.2) to (-3,0.5);
  \draw (1,.5) -- (1,2);
	\node at (-3, 0) {$\Pi_a$};
	\node at (-1, 0) {$\Pi_c$};
	\node at (1,0) {$\Pi_d$};
  \draw (-3,-2) -- (-3,-.5);
  \draw (0, -1.3) to (-1,-.5);
  \draw (0, -1.3) to (1,-.5);
  \draw (0, -1.3) to (0,-2);
	\node at (-3, -2.5) {$\Pi_a$};
	\node at (0, -2.5) {$\Pi_e$};
  \node at (2.5, 0) {$=$};
	\node at (4, 2.5) {$\Pi_b$};
	\node at (7, 2.5) {$\Pi_d$};
  \draw (4,.5) -- (4,2);
  \draw (7, 1.2) to (7,2);
  \draw (7, 1.2) to (8,0.5);
  \draw (7, 1.2) to (6,0.5);
	\node at (4, 0) {$\Pi_b$};
	\node at (6, 0) {$\Pi_c$};
	\node at (8,0) {$\Pi_e$};
  \draw (5, -1.3) to (4,-.5);
  \draw (5, -1.3) to (6,-.5);
  \draw (5, -1.3) to (5,-2);
  \draw (8,-2) -- (8,-.5);
	\node at (5, -2.5) {$\Pi_a$};
	\node at (8, -2.5) {$\Pi_e$};
\end{tikzpicture}
\end{equation}
Finally, given two $q$-admissible triples $(a,b,c)$ and $(d,b,c)$, we have that
\begin{equation} \label{constant multiple identity}
\begin{tikzpicture}[scale = 0.4] 
  \node at (0,1.5) {$\Pi_d$};
  \node at (-1,-1.5) {$\Pi_b$};
  \node at (1,-1.5) {$\Pi_c$};
  \draw (0, 0) to (0,1);
  \draw (0, 0) to (1,-1);
  \draw (0, 0) to (-1,-1);
  
  \node at (0,-4.5) {$\Pi_a$};
  \draw (0, -3) to (0,-4);
  \draw (0, -3) to (1,-2);
  \draw (0, -3) to (-1,-2);
  
  \node at (3, -1.5) {$=$};
  
  \node at (8, -1.5) {$
  	\begin{cases}
  	0, &\text{ if } a\neq d; \\
  	\frac{\theta(a,b,c)}{[a+1]_q} \id, &\text{ if } a = d
  	\end{cases}
  	$};
\end{tikzpicture},
\end{equation}
where $\frac{\theta(a,b,c)}{[a+1]_q} \in \bbC$ is some non-zero constant (see \cite[Definition 3.2.1]{chen2014temperleylieb}) whose actual value will not matter to us.

Throughout this paper, we will also use the symbols `` $\cap$ '', `` $\cup$ '', `` $\tri{}$ '' and `` $\itri{}$ '' to represent these $q$-admissible maps during calculations whenever the domain and codomain are clear from context.

\subsection{Chebyshev polynomials and fusion ring}
Let $\Delta_k(d)$ denote the polynomial (in variable $d$) defined by the following recurrence relation
\[
\Delta_0(d) = 1, \quad \Delta_1(d) = d, \quad \Delta_{k+1}(d) = d \Delta_k(d) - \Delta_{k-1}(d).
\]
%
The fusion rule \eqref{eq: fusion rule} allows us to identify the fusion ring of $\TLJ_n$ with
\[
K_0(\TLJ_n) \cong \Z[d]/\< \Delta_{n-1}(d) \>,
\]
where each basis element $[\Pi_k]$ is mapped to $\Delta_k(d)$.

\subsection{Deligne's tensor product} \label{subsec: Deligne tensor}
The following notion of tensor product of monoidal categories will be useful later on.
\begin{definition}
Let $\cC$ and $\cC'$ be fusion categories.
The \emph{(Deligne's) tensor product} $\cC \bigboxtimes \cC'$ is the fusion category with objects and morphism spaces given as follows:
\begin{enumerate}
\item Obj$(\cC \bigboxtimes \cC')$ consists of finite direct sums of the formal objects $A \boxtimes A'$, with $A$ and $A'$ objects in $\cC$ and $\cC'$ respectively; 
\item $\Hom_{\cC \boxtimes \cC'} \left(\bigoplus_i A_i \boxtimes A_i', \bigoplus_j B_j \boxtimes B_j'\right) = \bigoplus_{i,j} \Hom_\cC \left( A_i, B_j \right) \otimes \Hom_{\cC'} \left( A_i', B_j' \right)$.
\end{enumerate}
The monoidal structure on $\cC \bigboxtimes \cC'$ is defined component-wise by
\[
(A \boxtimes A') \otimes (B \boxtimes B') := (A \otimes B) \boxtimes (A' \otimes B'),
\] 
with monoidal unit $\1_\cC \boxtimes \1_{\cC'}$ and direct sums in each $\boxtimes$-component are distributive over $\boxtimes$.
The simple objects are given by $X \boxtimes X'$ for $X$ and $X'$ simple objects in $\cC$ and $\cC'$ respectively.
Given a fixed choice of simple representatives $\Irr(\cC)$ and $\Irr(\cC')$, the simple representatives of $\cC \bigboxtimes \cC'$ will be chosen as follows:
\[
\Irr(\cC \bigboxtimes \cC') := \{ A \boxtimes A' \mid A \in \Irr(\cC), A' \in \Irr(\cC')\}.
\]
\end{definition}
Given a finite set $S$ and fusion category $\cC_s$ for each $s \in S$, we define the fusion category $\bigboxtimesp{}{s \in S} \cC_s$  in a similar fashion.
It is easy to see that the fusion ring of $\bigboxtimesp{}{s \in S} \cC_s$ is just a tensor product of the individual components
\[
K_0 \left(\bigboxtimesp{}{s \in S} \cC_s \right) \cong \bigotimes_{s \in S} K_0(\cC_s),
\]
where the tensor products of the fusion rings are taken over $\Z$.
\begin{remark}
The Deligne's tensor product for abelian categories is by definition a universal object; when it exists, it is unique up to a unique isomorphism. 
The above ``naive'' tensor product of categories is just an explicit realisation of the Deligne's tensor product for fusion categories (which does not work for general tensor categories); see for example \cite[Chapter 5]{Row19}.
\end{remark}

\section{Coxeter quiver representations in fusion categories}\label{sec: Coxeter quiver and rep}

\subsection{Coxeter quivers}
Recall that a \emph{quiver} $\check{Q}$ is a directed graph.
Namely, it has a set of vertices $\check{Q}_0$ and a set of arrows $\check{Q}_1$ equipped with two maps $s, t: \check{Q}_1 \ra \check{Q}_0$ known as the source and target map respectively, providing the orientation.
A \emph{path} of a quiver (of length $n$) is a sequence of arrows $(\alpha_{n}, \alpha_{n-1}, \cdots, \alpha_1)$ satisfying $t(\alpha_i) = s(\alpha_{i+1})$.
Notice that by convention our path goes from ``right to left'':
\[\begin{tikzcd}
	\bullet & \bullet & {} & \cdots & {} & \bullet & \bullet
	\arrow[from=1-3, to=1-2, "\alpha_{n-1}"]
	\arrow[from=1-2, to=1-1, "\alpha_n"]
	\arrow[from=1-7, to=1-6, "\alpha_1"]
	\arrow[from=1-6, to=1-5, "\alpha_2"]
\end{tikzcd}\]
In this paper, all quivers are assumed to be acyclic and finite (both $\check{Q}_0$ and $\check{Q}_1$ are finite sets).

We are interested in the following generalised setting:
\begin{definition} \label{defn: Coxeter quiver}
A \emph{Coxeter quiver} $Q=(Q_0,Q_1,s,t,v)$ consists of a quiver $(Q_0,Q_1,s,t)$ equipped with an arrow label map $v: Q_1 \ra \{ n \in \mathbb{N} \mid n \geq 3 \}$. 
\end{definition}
\begin{example}
A Coxeter quiver $Q$ with $Q_0 = \{x,y,z\}$, $Q_1=\{\alpha, \alpha', \beta, \gamma\}$, $s(\alpha) = s(\alpha') = s(\gamma) = x$, $t(\alpha) = t(\alpha') = s(\beta) = y$ and $t(\beta) = t(\gamma) = z$ will be drawn as follows.
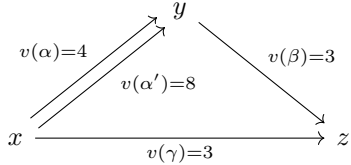
\[\begin{tikzcd}
	&& y \\
	\\
	x &&&& z
	\arrow["{v(\alpha)=4}", shift left=1, from=3-1, to=1-3]
	\arrow["{v(\alpha')=8}"'{pos=0.6}, shift right=1, from=3-1, to=1-3]
	\arrow["{v(\beta)=3}", from=1-3, to=3-5]
	\arrow["{v(\gamma) = 3}"', from=3-1, to=3-5]
\end{tikzcd}\]
When we do not wish to name the vertices and arrows, we shall use bullets for the vertices and simply write the labels above the arrows, where we use the ``Coxeter graph convention'' that an arrow without label is always assumed to have label `3'.
For example, the Coxeter quiver above shall be drawn as:
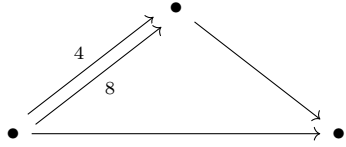
\[\begin{tikzcd}
	&& \bullet \\
	\\
	\bullet &&&& \bullet
	\arrow["4", shift left=1, from=3-1, to=1-3]
	\arrow["8"', shift right=1, from=3-1, to=1-3]
	\arrow[from=1-3, to=3-5]
	\arrow[from=3-1, to=3-5]
\end{tikzcd}\]
In particular, a Coxeter quiver with only label `3' will look just like a classical quiver.
\end{example}

\begin{remark}
The reader who is familiar with Coxeter theory may wonder why the label $\infty$ is not included.
As it will become clear later during the consideration of representations (\cref{defn: coxeter quiver representation}), an arrow with label $n$ should be thought of as weighted by $2\cos(\pi/n)$; the object $\nPi[v(\alpha)]_{\scriptscriptstyle v(\alpha)-3}$ appearing in the definition is an object with quantum dimension (equivalently Perron-Frobenius dimension) equal to $2\cos(\pi/n)$.
As such, the $n \ra \infty$ case (where $2\cos(\pi/n) \ra 2$) should be represented by double arrows
\begin{tikzcd}
	\bullet & \bullet 
	\arrow[from=1-1, to=1-2, shift right = 1]
	\arrow[from=1-1, to=1-2, shift left = 1]
\end{tikzcd}.
This coincides with the classical convention where the affine $\hat{A}_1$ quiver is just the Kronecker quiver.
\end{remark}

\begin{remark}
Note that a Coxeter quiver is not the same as a valued quiver -- a valued quiver has its arrows labelled by a pair of positive integers.
One should think of Coxeter quivers as being associated to \emph{possibly non-integral}, but symmetric \emph{Coxeter} matrices; whereas valued quivers are associated to \emph{integral}, but only symmetrisable \emph{Cartan} matrices.
\end{remark}

\subsection{Representations of Coxeter quivers}\label{subsec: Coxeter quiver rep}
Let $Q=(Q_0,Q_1,s,t,v)$ be a Coxeter quiver and let $\img(v) \subset \{3,4,5,...\}$ be the image of $v$.
We shall denote
\[
\cC_n := 
	\begin{cases}
	\TLJ_n, &\text{when $n$ is even}; \\
	\TLJ^{even}_n, &\text{when $n$ is odd}
	\end{cases}
\]
with $\TLJ_n$ and $\TLJ_n^{even}$ defined in \cref{subsection: TLJ}.

We associate a fusion category $\cC(Q)$ to $Q$ defined by (see \cref{subsec: Deligne tensor} for definition of $\bigboxtimes$)
\[
\cC(Q) := 
\bigboxtimesp{}{n \in \img(v)} \cC_{n}.
\]
We use $\1_{\cC(Q)}$ to denote the monoidal unit of $\cC(Q)$, dropping the subscript whenever the category $\cC(Q)$ is clear from context.
To differentiate between the simple objects in different $\cC_{n}$, we shall use $\nPi_i$ to denote the simple objects in $\cC_{n}$.

Note that given $J \subseteq \img(v)$, $\bigboxtimes_{n \in J} \cC_n$ is naturally identified with a full fusion subcategory of $\cC(Q)$.
As such, for $A \in \bigboxtimes_{n \in J} \cC_n$, we shall abuse notation and also use $A$ to denote its corresponding object in $\cC(Q)$ under the embedding.
In particular, $\nPi_i$ will also be used to denote the simple object $\nPi_i \boxtimes \left( \bigboxtimes_{j \in \img(v) \setminus \{n\}} \nPi[j]_0 \right) \in \cC(Q)$.

\begin{definition}\label{defn: coxeter quiver representation}
A \emph{representation} $V$ of a Coxeter quiver $Q$ consists of the following datum:
\begin{itemize}
\item for each vertex $i \in Q_0$, an object $V_i \in \cC(Q)$; and
\item for each arrow $i \xra{v(\alpha)} j$ in $Q$, a morphism $V_\alpha$ in $\cC(Q)$
\[
V_\alpha: \nPi[v(\alpha)]_{v(\alpha)-3} \otimes V_i \ra V_j.
\]
\end{itemize}
A \emph{morphism of representations} $f: V \ra W$ is a collection of map $f_i: V_i \ra W_i$ for each vertex $i$ such that $f_j \circ V_\alpha = W_\alpha \circ (\id \otimes f_i)$ for each arrow $i \xra{v(\alpha)} j$.
Together the objects and morphisms form the \emph{category of representations of $Q$}, which we denote by $\rep(Q)$.
\end{definition}
Note that when $Q$ only has label `3', we have $\cC(Q) = \TLJ_3^{even} \cong \vec_{\bbC}$ and the definition above coincides with the usual definition of representations (and morphisms) of classical quiver representations (over $\bbC$).

As in the case for classical quivers, $\rep(Q)$ is an abelian category induced by the abelian structure of the underlying $\cC(Q)$: the direct sums, kernels and cokernels are taken component wise.
In particular, a representation is also said to be \emph{indecomposable} if it can not be given as a direct sum of two non-trivial representations (up to isomorphism).
Moreover, each object $A$ of $\cC(Q)$ acts on a representation $V$ by tensoring (on the right) component wise: on each vertex $i$ the resulting representation $V\otimes A$ has $(V \otimes A)_i := V_i \otimes A$ and on each arrow $i \xra{v(\alpha)} j$ we now have $(V\otimes A)_\alpha := V_\alpha \otimes \id_A$.
Similarly, $A$ acts on morphisms $f: V \ra W$ by sending each $f_i$ to $f_i \otimes \id_A$.
In other words, $\rep(Q)$ is a right \emph{module category} over $\cC(Q)$ (see \cref{defn: module category}).
\begin{example}\label{eg:I2(5)Coxeterquiver}
Consider the Coxeter quiver $Q:= \begin{tikzcd}[column sep=large] x \ar[r, "v(\alpha)=5"] & y \end{tikzcd}$.
In this case, $\cC(Q) = \TLJ^{even}_5$ has two simple objects $\1 := \Pi_0$ and $\Pi := \Pi_2$, with fusion rule
\[
\Pi \otimes \Pi \cong \1 \oplus \Pi.
\]
(The category $\cC(Q)$ is also known as the Fibonacci category.)
A representation $V$ of $Q$ consists of a morphism $V_\alpha: \Pi \otimes V_x \ra V_y$ in $\cC(Q)$, where $V_x$ and $V_y$ are objects in $\cC(Q)$.
For example, the morphism $\cap: \Pi \otimes \Pi \ra \1$ given by the counit of the self-duality of $\Pi$ defines a representation $V$ with $V_\alpha = \cap$, $V_x = \Pi$ (not $\Pi \otimes \Pi$) and $V_y = \1$. 
Note that $V$ is indecomposable since both $V_x = \Pi$ and $V_y = \1$ are indecomposable and $V_\alpha = \cap \neq 0$.
We can also apply the action of $\Pi \in \cC(Q)$ to obtain other representations.
For example, with the representation $V$ above, $V \otimes \Pi$ is the representation defined by the morphism 
\[
\cap \otimes \id_\Pi: (\Pi \otimes \Pi) \otimes \Pi = \Pi \otimes (\Pi \otimes \Pi) \ra \1 \otimes \Pi = \Pi,
\]
so that $(V\otimes \Pi)_x = \Pi \otimes \Pi$, $(V\otimes \Pi)_y = \1 \otimes \Pi = \Pi$ and $(V\otimes \Pi)_\alpha = \cap \otimes \id_\Pi$.
\end{example}

\begin{remark}
As we shall see in \cref{subsec: path algebra}, $\rep(Q)$ is equivalent to the category of modules over a path algebra associated to $Q$ (\cref{thm: rep and module equiv}), just as in the classical case.
\end{remark}

\section{Unfolding and Gabriel's classification for Coxeter quivers}\label{sec: unfolding and gabriel classification}
In this section we shall prove our first main theorem, namely a classification of Coxeter quivers which has only finitely many indecomposable representations (up to isomorphism).
This will be an easy consequence of a relation between representations of Coxeter quivers and representations of their unfolded classical quivers.
Throughout, $Q = (Q_0, Q_1, s, t, v)$ denotes a Coxeter quiver.

\subsection{Aside on \texorpdfstring{$\cC(Q)$}{C(Q)}} \label{subsec: C(Q) properties}
Let us describe some properties of $\cC(Q)$ that we will require later on.
As with $\TLJ_n$ and $\TLJ^{even}_n$, the representatives of isomorphism classes of simple objects in $\cC(Q)$ are given by:
\[
\Irr(\cC(Q)) = \left\{ \bigboxtimes_{n \in \img(v)} A(n) \mid A(n) \in \Irr(\cC_n) \right\}. 
\]
For $A = \bigboxtimes_{n \in \img(v)} A(n) \in \Irr(\cC(Q))$ and $k \in \img(v)$, we refer to $A(k) \in \cC_k$ as the \emph{$k$th (tensor) component} of $A$.
On the other hand, we define $\bar{A}(k)$ to be the ``complement'' of $A(k)$:
\[
\bar{A}(k) :=  \bigboxtimes_{n \in \img(v) \setminus \{k\}} A(n) \in \bigboxtimes_{n \in \img(v) \setminus \{k\}} \cC_n.
\]
In particular, we get that
\[
A = A(k) \otimes \bar{A}(k) \in \cC(Q),
\]
where $A(k)$ and $\bar{A}(k)$ are both viewed as objects in $\cC(Q)$ through the natural embedding of $\cC_k$ and $\bigboxtimes_{n \in \img(v) \setminus \{k\}} \cC_n$ inside $\cC(Q)$ respectively.
Note that the equation above is indeed an equality since $\cC_n$, and hence $\cC(Q)$, are strict monoidal categories.

Let $B,C\in \Irr(\cC(Q))$.
By the definition of $\cC(Q)$ we have that
\[
\Hom_{\cC(Q)}(\nPi[k]_{k-3} \otimes B, C) = \Hom_{\cC_k}(\nPi[k]_{k-3} \otimes B(k), C(k)) \otimes \left(\bigotimes_{n \in \img(v) \setminus \{k\}} \Hom_{\cC_n}(B(n), C(n)) \right).
\]
Notice that the following conditions are all equivalent:
\begin{equation}\label{eqn:equivcondition}
\begin{cases}
C \overset{\oplus}{\subseteq} \nPi_{n-3} \otimes B;\\
\Hom_{\cC(Q)}(\nPi_{n-3} \otimes B, C) = \bbC\{\tri \otimes \id_{\bar{B}(n)}\} \neq 0; \\
C(n) \overset{\oplus}{\subseteq} \nPi_{n-3} \otimes B(n) \text{ and } \bar{B}(n) = \bar{C}(n); \text{ and }\\
\Hom_{\cC_n}(\nPi_{n-3} \otimes B(n), C(n)) = \bbC\{\tri\} \neq 0 \text{ and } \bar{B}(n) = \bar{C}(n),
\end{cases}
\end{equation}
where $\tri$ denotes the $q$-admissible morphism in $\Hom_{\cC_n}(\nPi_{n-3} \otimes B(n), C(n))$.
In particular, we have the following variant of \eqref{eqn: hom space trivalent} for $\cC(Q)$:
\begin{equation} \label{eqn: C(Q) hom space tri}
\Hom_{\cC(Q)}(\nPi_{n-3} \otimes B, C) = 
\begin{cases}
\bbC\{\tri \otimes \id_{\bar{B}(n)}\}, &\text{ if \eqref{eqn:equivcondition} holds};\\
0,  &\text{ otherwise}.
\end{cases}
\end{equation}

\subsection{From Coxeter quivers to classical quivers} \label{subsec: unfolding}
To each Coxeter quiver $Q$, we shall assign an \emph{unfolded} classical quiver $\check{Q} = (\check{Q}_0, \check{Q}_1,s,t)$ as follows:
\begin{itemize}
\item its set of vertices $\check{Q}_0$ is given by the set $\Irr(\cC(Q)) \times Q_0$;
\item Given an arrow $i \xra{v(\alpha)} j$ in $Q$, we construct an arrow $(B,i) \xra{\beta} (C,j)$ for each pair of simple objects $B$ and $C$ satisfying the condition that $C \overset{\oplus}{\subseteq} \nPi[v(\alpha)]_{v(\alpha) -3} \otimes B$ (or any of the equivalent conditions in \eqref{eqn:equivcondition}).
We call these arrows the \emph{unfolded arrows of $\alpha$} and denote the set of unfolded arrows of $\alpha$ by $\ufld(\alpha)$.
The set of arrows $\check{Q}_1$ of $\check{Q}$ is then defined to be the disjoint union over all sets of unfolded arrows: $\check{Q}_1 = \coprod_{\alpha \in Q_1} \ufld(\alpha)$; two arrows in $\check{Q}$ between the same vertices obtained from unfoldings of different arrows in $Q$ are thus by convention not the same arrows. 
\end{itemize} 
Refer to \cref{fig: unfolding Coxeter quivers} and \cref{eg:unfoldedA4} for some examples of unfoldings.
Notice that even if $Q$ is connected (i.e.\ its underlying quiver is connected), $\check{Q}$ need not be connected.

\begin{figure}[ht]
\adjustbox{scale=0.8,center}{%
\begin{tikzcd}
	& {(\nPi[3]_0\boxtimes\nPi[5]_0, y)} & {(\nPi[3]_0\boxtimes\nPi[5]_0, z)} & {(\nPi[3]_0\boxtimes\nPi[5]_0, w)} \\
	{(\nPi[3]_0\boxtimes\nPi[5]_2, x)} &&&&& {(\nPi[4]_2, x)} & {(\nPi[4]_2, y)} & {(\nPi[4]_2, z)} \\
	& {(\nPi[3]_0\boxtimes\nPi[5]_2, y)} & {(\nPi[3]_0\boxtimes\nPi[5]_2, z)} & {(\nPi[3]_0\boxtimes\nPi[5]_2, w)} && {(\nPi[4]_1, x)} & {(\nPi[4]_1, y)} & {(\nPi[4]_1, z)} \\
	{(\nPi[3]_0\boxtimes\nPi[5]_0, x)} &&&&& {(\nPi[4]_0, x)} & {(\nPi[4]_0, y)} & {(\nPi[4]_1, z)} \\
	\\
	x & y & z & w && x & y & z
	\arrow["5", from=6-1, to=6-2]
	\arrow[from=6-2, to=6-3]
	\arrow[from=6-3, to=6-4]
	\arrow[from=4-1, to=3-2]
	\arrow[from=2-1, to=3-2]
	\arrow[from=2-1, to=1-2]
	\arrow[from=1-2, to=1-3]
	\arrow[from=1-3, to=1-4]
	\arrow[from=3-2, to=3-3]
	\arrow[from=3-3, to=3-4]
	\arrow["4", from=6-6, to=6-7]
	\arrow["4"', from=6-8, to=6-7]
	\arrow[from=4-6, to=3-7]
	\arrow[from=2-6, to=3-7]
	\arrow[from=3-6, to=2-7]
	\arrow[from=3-6, to=4-7]
	\arrow[from=4-8, to=3-7]
	\arrow[from=2-8, to=3-7]
	\arrow[from=3-8, to=2-7]
	\arrow[from=3-8, to=4-7]
\end{tikzcd}
}
\caption{Two examples of unfoldings $\check{Q}$ of Coxeter quivers $Q$, with $\check{Q}$ above $Q$.}
\label{fig: unfolding Coxeter quivers}
\end{figure}

Given an arrow $i \xra{v(\alpha)} j$ in $Q$, its set of unfolded arrows $\ufld(\alpha)$ can be equivalently described as follows, depending on the parity of $n:= v(\alpha)$.
When $n$ is even, $\ufld(\alpha)$ has the following $2(n-2)$ arrows for each simple object $\bar{A}$ in $\Irr\left( \bigboxtimes_{k \in \img(v)\setminus \{n\}} \cC_k \right)$:
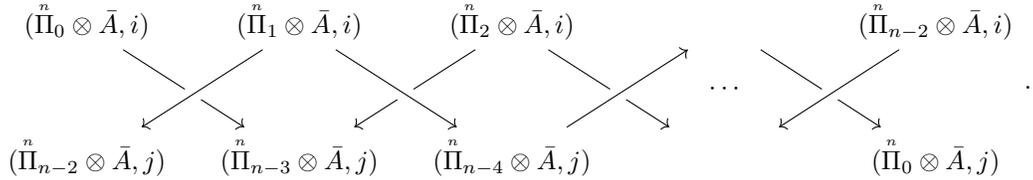
\[\begin{tikzcd}[row sep = small, column sep = small]
	{{(\nPi_0 \otimes \bar{A}, i)}} & {{(\nPi_1 \otimes \bar{A}, i)}} & {{(\nPi_{2} \otimes \bar{A}, i)}} & {\phantom{(\nPi_{3} \otimes \bar{A}, i)}} & {{(\nPi_{n-2} \otimes \bar{A}, i)}} \\
	&&& \cdots \\
	{{(\nPi_{n-2} \otimes \bar{A}, j)}} & {{(\nPi_{n-3} \otimes \bar{A}, j)}} & {{(\nPi_{n-4} \otimes \bar{A}, j)}} & {\phantom{(\nPi_{n-5} \otimes \bar{A}, j)}} & {{(\nPi_{0} \otimes \bar{A}, j)}}
	\arrow[from=1-1, to=3-2, crossing over]
	\arrow[from=1-3, to=3-2, crossing over]
	\arrow[from=1-3, to=3-4, crossing over]
	\arrow[from=1-4, to=3-5, crossing over]
	\arrow[from=1-5, to=3-4, crossing over]
	\arrow[from=3-3, to=1-4, crossing over]
	\arrow[from=1-2, to=3-1, crossing over]
	\arrow[from=1-2, to=3-3, crossing over]
\end{tikzcd}.\]
When $n$ is odd, $\ufld(\alpha)$ has instead the following $(n-2)$ arrows for each $\bar{A}$ in $\Irr\left( \bigboxtimes_{k \in \img(v)\setminus \{n\}} \cC_k \right)$:
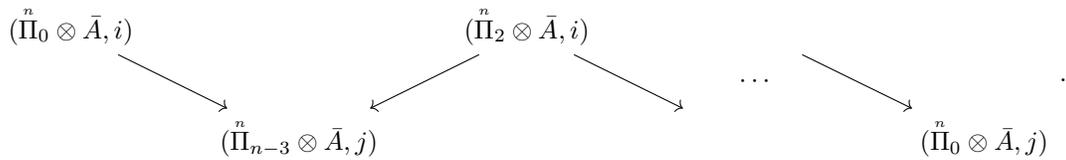
\[\begin{tikzcd}[row sep = tiny]
	{{(\nPi_0 \otimes \bar{A}, i)}} && {{(\nPi_{2} \otimes \bar{A}, i)}} & {\phantom{(\nPi_{3} \otimes \bar{A}, i)}} \\
	&&& \cdots \\
	& {{(\nPi_{n-3} \otimes \bar{A}, j)}} && {\phantom{(\nPi_{n-5} \otimes \bar{A}, j)}} & {{(\nPi_{0} \otimes \bar{A}, j)}}
	\arrow[from=1-1, to=3-2]
	\arrow[from=1-3, to=3-2]
	\arrow[from=1-3, to=3-4]
	\arrow[from=1-4, to=3-5]
\end{tikzcd}.\]

\subsection{Abelian equivalence between \texorpdfstring{$\rep(Q)$ and $\rep(\check{Q})$}{rep(Q) and its unfolding}}
Let $Q$ be a Coxeter quiver and $\check{Q}$ its unfolded quiver.
We shall also denote the category of representations of $\check{Q}$ (over $\bbC$) by the same notation $\rep(\check{Q})$.
We claim that the two abelian categories $\rep(Q)$ and $\rep(\check{Q})$ are equivalent.
Before we prove this let us set up some notation.

Throughout this subsection, we simplify our notation and use $\cC:= \cC(Q)$ and $I := \Irr(\cC)$.
Since $\cC$ is semi-simple, every object $X \in \cC$ is isomorphic to a direct sum of the simple objects in $I$.
This is called the \emph{semi-simple decomposition} of $X$ (with respect to $I$).
For each simple object $A \in I$, we use $A(X)$ to denote the direct sum of \emph{all} $A$ (i.e.\ with multiplicity) appearing in the semi-simple decomposition of $X$.
In particular, the semi-simple decomposition of $X$ is given by
\[
X \cong \bigoplus_{A \in I} A(X).
\]
While the isomorphism realising the semi-simple decomposition of $X$ may not be unique, note that the decomposition above is unique.
As such, we shall use $|A(X)|$ to denote the number of copies of $A$ in $|A(X)|$, so that $A(X) = A^{\oplus |A(X)|}$.

We are now ready to prove our theorem.
\begin{theorem} \label{thm: rep Q and checkQ equiv}
There exists an exact functor $\cU: \rep(Q) \ra \rep(\check{Q})$ that is fully faithful and essentially surjective. In other words, $\rep(Q)$ and $\rep(\check{Q})$ are equivalent as abelian categories.
\end{theorem}
\begin{proof}
Throughout this proof, let us fix the following datum, which we will use to construct the functor $\cU$.
\begin{itemize}
\item For each object $X$ in $\cC$, we fix isomorphisms $\varphi_X: X \rightleftarrows \bigoplus_{A \in I} A(X) : \varphi_X^{-1}$ identifying $X$ with its semi-simple decomposition.
\end{itemize}
A hook arrow $\hookrightarrow$ and $\iota$ will always denote the canonical inclusion of a summand into a direct sum containing it; on the other hand a two-headed arrow $\twoheadrightarrow$ and $\rho$ will always denote the canonical projection of a direct sum onto one of its summand. 
It will be clear from the context which summand and direct sum we are referring to.

We shall construct the functor $\cU: \rep(Q) \ra \rep(\check{Q})$ as follows.
\begin{itemize}
\item On objects: \\
	Given a representation $V$ of $Q$, we will define a representation $U:= \cU(V)$ of $\check{Q}$. 
	On each vertex $(A,i) \in \check{Q}_0$, we set $U_{(A,i)} := \bbC^{|A(V_i)|}$.
	Let $i \xra{v(\alpha)} j$ be an arrow of $Q$ and denote $n := v(\alpha)$.
	Using the semi-simple decomposition of $V_i$ and $V_j$, the morphism $V_{\alpha}: \nPi_{n-3} \otimes V_i \ra V_j$ induces the morphism $\tilde{V}_\alpha$ defined by the composition:
	\[
	\tilde{V}_\alpha: \bigoplus_{A\in I} \nPi_{n-3} \otimes A(V_i) = \nPi_{n-3} \otimes \left(\bigoplus_{A\in I} A(V_i) \right) \xra{\id \otimes \varphi_{V_i}^{-1}} \nPi_{n-3} \otimes V_i \xra{V_\alpha} V_j \xra{\varphi_{V_j}} \bigoplus_{A \in I} A(V_j).
	\]
	Note that $\tilde{V}_\alpha$ is completely determined by the component morphisms $(\tilde{V}_\alpha)_{B,C}$ for each pair of simples $B,C$ such that $C \overset{\oplus}{\subseteq} \nPi_{n-3}\otimes B$, which by the construction of $\check{Q}$ is equivalent to having
	an arrow $(B, i) \xra{\beta} (C,j)$ in $\check{Q}$ that lives in $\ufld(\alpha)$ -- the other component morphisms are all zero.
	Moreover, for each such pair the morphism $(\tilde{V}_\alpha)_{B,C}$ is by definition a $|B(V_i)| \times |C(V_j)|$ matrix with each entry an element in $\Hom_{\cC}(\nPi_{n-3} \otimes B, C) = \bbC\{ \tri \otimes \id_{\bar{B}(n)} \}$ (see \eqref{eqn: C(Q) hom space tri}).
	We define $U_\beta$ to be the image of $(\tilde{V}_\alpha)_{B,C}$ in $M_{|B(V_i)| \times |C(V_j)|}(\bbC) = \Hom_{\bbC}(\bbC^{|B(V_i)|}, \bbC^{|C(V_j)|})$ by sending $\tri \otimes \id_{\bar{B}(n)} \mapsto 1 \in \bbC$.
\item On morphisms:\\
	Let $f: V \ra W$ be a morphism of representations of $Q$.
	For each $f_i: V_i \ra W_i$, define 
	\[
	\tilde{f_i}:= \varphi_{W_i} \circ f_i \circ \varphi_{V_i}^{-1} \in \Hom_\cC \left(\bigoplus_{A \in I} A(V_i), \bigoplus_{A \in I} A(W_i) \right).
	\]
	For each simple object $B \in I$, we simplify our notation to $(\tilde{f_i})_B := (\tilde{f_i})_{B,B}$ to denote the component morphism of $\tilde{f_i}$ that maps between the summands $B(V_i) \ra B(W_i)$ (which uniquely determine $\tilde{f_i}$ since all other component morphisms are zero).
	The morphism $(\tilde{f_i})_B$ is by definition a $|B(V_i)| \times |B(W_i)|$ matrix with each entry an element of $\Hom_{\cC}(B,B) = \bbC\{\id_B\}$.
	We define $(\cU(f))_{(B,i)}: \bbC^{|B(V_i)|} \ra \bbC^{|B(W_i)|}$ to be the image of $(\tilde{f_i})_B$ in $M_{|B(V_i)| \times |B(W_i)|}(\bbC) = \Hom_{\bbC}(\bbC^{|B(V_i)|}, \bbC^{|B(W_i)|})$ that sends $\id_B \mapsto 1 \in \bbC$.
\end{itemize}
	Let us first show that $\cU(f)$ does indeed define a morphism of representations. 
	Namely for each arrow $(B, i) \xra{\beta} (C,j)$ in $\check{Q}$ living in $\ufld(\alpha)$ with $n:= v(\alpha)$, we have to show that the following diagram commmutes:
	\[
	\begin{tikzcd}
	\bbC^{|B(V_i)|} = \cU(V)_{(B,i)} \arrow["\cU(V)_{\beta}"]{r} \arrow["\cU(f)_{(B,i)}"']{d} & \cU(V)_{(C,j)} = \bbC^{|C(V_j)|} \arrow["\cU(f)_{(C,j)}"]{d}\\
	\bbC^{|B(W_i)|} = \cU(W)_{(B,i)} \arrow["\cU(W)_{\beta}"]{r} & \cU(W)_{(C,j)} = \bbC^{|C(W_j)|}.
	\end{tikzcd}
	\]
	By the construction of $\cU$, the commutativity of this diagram is equivalent to the commutativity of the (outermost) square diagram formed by the four dashed arrows below:
\[\begin{tikzcd}[column sep=small, row sep=tiny]
	{\nPi_{n-3} \otimes B(V_i)} &&&&&& {C(V_j)} \\
	&&& {\text{(a)}} \\
	&& {\nPi_{n-3} \otimes V_i} && {V_j} \\
	& {\text{(c)}} && {\text{(e)}} && {\text{(d)}} \\
	&& {\nPi_{n-3} \otimes W_i} && {W_j} \\
	&&& {\text{(b)}} \\
	{\nPi_{n-3} \otimes B(W_i)} &&&&&& {C(W_j).}
	\arrow["{(\tilde{V}_\alpha)_{B,C}}", dashed, from=1-1, to=1-7]
	\arrow["{\id \otimes (\varphi_{V_i}^{-1} \circ \iota)}", from=1-1, to=3-3]
	\arrow["{\id \otimes (\tilde{f_i})_B}", dashed, from=1-1, to=7-1]
	\arrow["{(\tilde{f_j})_C}", dashed, from=1-7, to=7-7]
	\arrow["{V_\alpha}", from=3-3, to=3-5]
	\arrow["{\id \otimes f_i}", from=3-3, to=5-3]
	\arrow["{f_j}", from=3-5, to=5-5]
	\arrow["{\rho \circ \varphi_{V_j}}", from=3-5, to=1-7]
	\arrow["{W_\alpha}", from=5-3, to=5-5]
	\arrow["{\rho \circ \varphi_{W_j}}", from=5-5, to=7-7]
	\arrow["{\id \otimes (\varphi_{W_i}^{-1} \circ \iota)}", from=7-1, to=5-3]
	\arrow["{(\tilde{W}_\alpha)_{B,C}}", dashed, from=7-1, to=7-7]
\end{tikzcd}\]
	We claim that all five subdiagrams are commutative.
	Indeed, subdiagram (e) is commutative since $f$ is a morphism of representations;
	the subdiagrams (a) and (b)
	are commutative by the definition of the dashed maps.
	Finally, note that the following diagrams are commutative
	\begin{center}
	\begin{tikzcd}
	B(V_i) \ar[rd, "\varphi_{V_i}^{-1} \circ \iota" {yshift = -5pt}] \ar[ddd, dashed, "(\tilde{f_i})_B"] & \\
	& V_i  \ar[d, "f_i"]  \\
	& W_i  \\
	B(W_i) \ar[ru, "\varphi_{W_i}^{-1} \circ \iota"' {yshift = 5pt}] & 
	\end{tikzcd}
	,
	\begin{tikzcd}
	& C(V_j) \ar[ddd, dashed, "(\tilde{f_j})_C"] \\
	V_j \ar[d, "f_j"] \ar[ru,"\rho \circ \varphi_{V_j}"{yshift = -3pt}] & \\
	W_j \ar[rd,"\rho \circ \varphi_{W_j}"' {yshift = 3pt}] & \\
	& C(W_j)
	\end{tikzcd}
	\end{center}
	since $(\tilde{f}_i)_B$ (resp.\ $(\tilde{f}_j)_C$) is the restriction of $\tilde{f}_i$ (resp.\ $\tilde{f}_j$) to the summands $B(V_i) \ra B(W_i)$ (resp.\ $C(V_j) \ra C(W_j)$).
	The diagram on the right is exactly the subdiagram (d), and by applying the tensor product functor $\nPi_{n-3} \otimes -$ to the left diagram above we obtain subdiagram (c).
	This shows that all the subdiagrams are commutative and thus the dashed square commutes as desired.
	
	The fact that $\cU$ sends identity to identity follows easily from the definition.
	Moreover, $\cU(g_i)\circ\cU(f_i) = \cU(g_i \circ f_i)$ for all $i$ since $(\tilde{g_i})_B \circ (\tilde{f_i})_B = (\widetilde{g_i\circ f_i})_B$ for all simple object $B$.
	This completes the proof that $\cU$ is a functor from $\rep(Q)$ to $\rep(\check{Q})$.
	
The fact that $\cU$ is essentially surjective is clear from construction. 
Using the vector space structure on hom spaces, it is also easy to see that $\cU$ is faithful.

To show fullness of $\cU$, let $g \in \Hom_{\rep(\check{Q})}(\cU(V), \cU(W))$ for $V,W$ representations of $Q$, so that we have $g_{(A,i)} : \bbC^{|A(V_i)|} \ra \bbC^{|A(W_i)|}$ for each simple object $A \in I$ and vertex $i \in Q_0$.
For each vertex $i\in Q_0$ we define $f_i: V_i \ra W_i$ to be the unique morphism determined by
\[
(\tilde{f}_i)_{A} := g_{(A,i)}\cdot \id_A : A(V_i) \ra A(W_i)
\]
for each simple object $A$.
It is clear that if $f = \{f_i\}_{i \in {Q_0}}$ form a morphism of representations of $Q$, then $g$ would be equal to $\cU(f)$ by definition.
As such, we are left to show that $f$ is indeed a morphism of representations of $Q$; namely we are required to show that the diagram below commutes for each arrow $i \xra{v(\alpha)} j$ with $n:= v(\alpha)$
	\[
	\begin{tikzcd}
	\nPi_{n-3} \otimes V_i \ar[r, "V_\alpha"] \ar[d, "\id \otimes f_i"] & V_j \ar[d, "f_j"] \\
	\nPi_{n-3} \otimes W_i \ar[r, "W_\alpha"] & W_j.
	\end{tikzcd}
	\]
Using the semi-simple decomposition of $V_i, V_j, W_i$ and $W_j$, we can instead show that diagram of morphisms between the induced direct sum decomposition is commutative:
\begin{equation} \label{eqn: comm square for unfold fullness}
	\begin{tikzcd}
	\bigoplus_{A \in I} \nPi_{n-3} \otimes A(V_i) \ar[r, "\tilde{V}_\alpha"] \ar[d, "\id \otimes \tilde{f}_i"] & \bigoplus_{A \in I} A(V_j) \ar[d, "\tilde{f}_j"] \\
	\bigoplus_{A \in I} \nPi_{n-3} \otimes A(W_i) \ar[r, "\tilde{W}_\alpha"] & \bigoplus_{A \in I} A(W_j).
	\end{tikzcd}
\end{equation}
Let us start by analysing the composition given by the top right corner:
	\[
	\begin{tikzcd}
	\bigoplus_{A \in I} \nPi_{n-3} \otimes A(V_i) \ar[r, "\tilde{V}_\alpha"] & \bigoplus_{A \in I} A(V_j) \ar[d, "\tilde{f}_j"] \\
	 & \bigoplus_{A \in I} A(W_j).
	\end{tikzcd}
	\]
The morphism $\tilde{f}_j$ is uniquely determined by the component $(\tilde{f}_j)_A$, with all other component morphisms being zero.
On the other hand, $\tilde{V}_{\alpha}$ is determined by the component morphisms $(\tilde{V}_{\alpha})_{B,C}$ for each pair of simples $B,C$ such that $C \overset{\oplus}{\subseteq} \nPi_{n-3} \otimes B$, which is equivalent to having an arrow $(B,i) \xra{\beta} (C,j)$ in $\ufld(\alpha)$.
In total, we have that the composition $\tilde{f}_j \circ \tilde{V}_\alpha$ is uniquely determined by $(\tilde{f}_j \circ \tilde{V}_\alpha)_{B,C}$ for each pair of simples $B,C$ with $(B,i) \xra{\beta} (C,j)$ an arrow in $\ufld(\alpha)$, and
\[
(\tilde{f}_j \circ \tilde{V}_\alpha)_{B,C} = (\tilde{f}_j)_C \circ (\tilde{V}_\alpha)_{B,C}.
\]
By a similar argument, the composition $\tilde{W}_\alpha \circ (\nPi_{n-3} \otimes \tilde{f}_i)$ given by the bottom right corner 
	\[
	\begin{tikzcd}
	\bigoplus_{A \in I} \nPi_{n-3} \otimes A(V_i) \ar[d, "\id \otimes \tilde{f}_i"] &  \\
	\bigoplus_{A \in I} \nPi_{n-3} \otimes A(W_i) \ar[r, "\tilde{W}_\alpha"] & \bigoplus_{A \in I} A(W_j).
	\end{tikzcd}
	\]
is uniquely determined by $(\tilde{W}_\alpha \circ (\nPi_{n-3} \otimes \tilde{f}_i))_{B,C}$ for each pair of simples $B,C$ with $(B,i) \xra{\beta} (C,j)$ an arrow in $\ufld(\alpha)$, and
\[
(\tilde{W}_\alpha \circ (\nPi_{n-3} \otimes \tilde{f}_i))_{B,C} = (\tilde{W}_\alpha)_{B,C} \circ (\nPi_{n-3} \otimes (\tilde{f}_i)_B).
\]
As such, the square in \eqref{eqn: comm square for unfold fullness} commutes if and only if $(\tilde{f}_j)_C \circ (\tilde{V}_\alpha)_{B,C} = (\tilde{W}_\alpha)_{B,C} \circ (\nPi_{n-3} \otimes (\tilde{f}_i)_B)$ for each pair of simples $B,C$ with $(B,i) \xra{\beta} (C,j)$ an arrow in $\ufld(\alpha)$, which amounts to asking for commutativity of the following square
\[
	\begin{tikzcd}[column sep = large]
	\nPi_{n-3} \otimes B(V_i) \ar[r, "(\tilde{V}_\alpha)_{B,C}"] \ar[d, "\id \otimes (\tilde{f}_i)_B"] & C(V_j) \ar[d, "(\tilde{f}_j)_C"] \\
	\nPi_{n-3} \otimes B(W_i) \ar[r, "(\tilde{W}_\alpha)_{B,C}"] & C(W_j).
	\end{tikzcd}
\]
The commutativity of the above square is a direct consequence of the commutativity of the following square:
	\[
	\begin{tikzcd}
	\bbC^{|B(V_i)|} \arrow["\cU(V)_{\beta}"]{r} \arrow["g_{(B,i)}"']{d} & \bbC^{|C(V_j)|} \arrow["g_{(C,j)}"]{d}\\
	\bbC^{|B(W_i)|} \arrow["\cU(W)_{\beta}"]{r} & \bbC^{|C(W_j)|},
	\end{tikzcd}
	\]
which is given to us by the fact that $g$ is a morphism of representations of $\check{Q}$.
This completes the proof that $f$ is morphism of representations of $Q$, which shows that $\cU$ is full.

Finally the exactness of $\cU$ follows directly from the semi-simplicity of both $\cC(Q)$ and $\vec_{\bbC}$.
\end{proof}

\begin{example}\label{eg:unfoldedA4}
We continue with the Coxeter quiver $Q$ from \cref{eg:I2(5)Coxeterquiver}.
The associated unfolded quiver $\check{Q}$ is given by
\[
\check{Q} = \begin{tikzcd}[row sep = small]
	{(\Pi,x)} & {(\Pi,y)} \\
	{(\1, x)} & {(\1,y)}
	\arrow[from=2-1, to=1-2, crossing over]
	\arrow[from=1-1, to=1-2]
	\arrow[from=1-1, to=2-2]
\end{tikzcd},
\]
which is a type $A_4$ quiver.
As such, $\rep(Q) \cong \rep(\check{Q})$ has 10 indecomposable objects, in bijection with the 10 positive roots of the $A_4$ root system.
The 10 indecomposable representations in $\rep(Q)$ and their corresponding indecomposable representations in $\rep(\check{Q})$ in terms of their dimension vectors (each of which uniquely determines the indecomposable representation) are given as follows.
\begin{multicols}{2}
\begin{center}
\begin{enumerate}[(i)]
\item $0: \Pi \otimes \1 \ra 0 
		\ \leftrightsquigarrow \
	\begin{psmallmatrix} 0 & 0 \\ 1 & 0 \end{psmallmatrix}$; \\
\item $0: \Pi \otimes 0 \ra \1 
		\ \leftrightsquigarrow \
	\begin{psmallmatrix} 0 & 0 \\ 0 & 1 \end{psmallmatrix}$; \\
\item $\id_\Pi: \Pi \otimes \1 \ra \Pi
		\ \leftrightsquigarrow \
	\begin{psmallmatrix} 0 & 1 \\ 1 & 0 \end{psmallmatrix}$; \\
\item $\cap: \Pi \otimes \Pi \ra \1
		\ \leftrightsquigarrow \
	\begin{psmallmatrix} 1 & 0 \\ 0 & 1 \end{psmallmatrix}$; \\
\item $\tri : \Pi \otimes \Pi \ra \Pi
		\ \leftrightsquigarrow \
	\begin{psmallmatrix} 1 & 1 \\ 0 & 0 \end{psmallmatrix}$; \\
\item $0: \Pi \otimes \Pi \ra 0 
		\ \leftrightsquigarrow \ 
	\begin{psmallmatrix} 1 & 0 \\ 0 & 0 \end{psmallmatrix}$; \\
\item $0: \Pi \otimes 0 \ra \Pi
		\ \leftrightsquigarrow \
	\begin{psmallmatrix} 0 & 1 \\ 0 & 0 \end{psmallmatrix}$; \\
\item $\id_\Pi \otimes \id_\Pi: \Pi \otimes \Pi \ra \Pi\otimes \Pi
		\ \leftrightsquigarrow \
	\begin{psmallmatrix} 1 & 1 \\ 0 & 1 \end{psmallmatrix}$; \\
\item $\cap \otimes \id_\Pi: \Pi \otimes \Pi\otimes \Pi \ra \Pi
		\ \leftrightsquigarrow \
	\begin{psmallmatrix} 1 & 1 \\ 1 & 0 \end{psmallmatrix}$; \\
\item $\tri \otimes \id_\Pi: \Pi \otimes \Pi\otimes \Pi \ra \Pi\otimes\Pi
		\ \leftrightsquigarrow \
	\begin{psmallmatrix} 1 & 1 \\ 1 & 1 \end{psmallmatrix}$.
\end{enumerate}
\end{center}
\end{multicols}
\end{example}

\subsection{Gabriel's classification for Coxeter quivers}
Following the classical notion, we make the following definition:
\begin{definition}
A Coxeter quiver $Q$ is of \emph{finite type} if $\rep(Q)$ has only finitely many indecomposable objects up to isomorphism.
\end{definition}
The following observation will be useful in the proof of Gabriel's theorem for Coxeter quivers later on.
Let $Q = (Q_0, Q_1, s, t, v)$ be a Coxeter quiver.
A \emph{Coxeter subquiver} $Q'$ of $Q$ is just a subquiver $(Q'_0, Q'_1)$ (namely $Q'_0 \subseteq Q_0$ and $Q'_1 \subseteq Q_1$ with $s(\alpha), t(\alpha) \in Q'_0$ for all $\alpha \in Q'_1$) equipped with the restricted label map $v':= v|_{Q'_1}$.
In particular, $\img(v') \subseteq \img(v)$ and so the associated fusion category $\cC(Q')$ is naturally identified with a fusion subcategory of $\cC(Q)$.
It follows easily that this identifies $\rep(Q')$ with an abelian subcategory of $\rep(Q)$.
As such, if there exists a Coxeter subquiver of $Q$ that is not of finite type, then $Q$ can not be of finite type as well.

We now state the crucial lemma:
\begin{lemma} \label{lem: multilaced quiver not finite type}
The Coxeter quiver 
$
Q':= \begin{tikzcd}
	i & j
	\arrow["m", shift left=1, from=1-1, to=1-2]
	\arrow["n"', shift right=1, from=1-1, to=1-2]
\end{tikzcd}
$
is not of finite type.
\end{lemma}
\begin{proof}
Using \cref{thm: rep Q and checkQ equiv} it is sufficient to show that its unfolded classical quiver $\check{Q}$ is not of finite type.
We prove this by showing that $\check{Q}$ contains a subquiver that is not of finite type.

Note that if $m=n$ then it follows by the construction of $\check{Q}$ that $\check{Q}$ will contain double arrows 
$\begin{tikzcd} (i, \1) \ar[shift left=1, from=1-1, to=1-2] \ar[shift right=1, from=1-1, to=1-2] & (j, \nPi_{n-3}) \end{tikzcd}$.
Thus for the remaining part of the proof we shall assume that $m \neq n$.

Let us start with the case where both $m, n \geq 4$.
In this case notice that
\[\begin{tikzcd}[column sep=small, row sep = small]
	{(\nPi[m]_{m-3}\otimes \nPi_2,j)} && {(\nPi[m]_{2}\otimes \nPi_{n-3},j)} \\
	& {(\nPi[m]_{m-3}\otimes \nPi_{n-3},i)} \\
	{(\nPi[m]_{m-3}\otimes \nPi_0,j)} && {(\nPi[m]_{0}\otimes \nPi_{n-3},j)}
	\arrow[from=2-2, to=1-1]
	\arrow[from=2-2, to=3-1]
	\arrow[from=2-2, to=1-3]
	\arrow[from=2-2, to=3-3]
\end{tikzcd}\]
is a subquiver of $\check{Q}$, hence $Q$ is not of finite type.

We are now left with the case where at least one of $m,n$ is 3; without loss of generality suppose $n = 3$ and $m \geq 4$.
We shall denote the arrow with label $m$ as $\alpha$ and the arrow with label $n=3$ as $\beta$.
Our analysis will depend the parity of $m$.
\begin{itemize}
\item (odd case) $m = 2k+1$ for $k \geq 2$.\\
	Denote $$\ell := \begin{cases}
		k,   &\text{ if $k$ is even}; \\
		k-1, &\text{ if $k$ is odd}.
		\end{cases}$$
	Using \eqref{eq: z/2 fusion rule} and \eqref{eq: 1 b fusion rule} we get $\nPi[m]_{\ell} \otimes \nPi[m]_{m-3} \cong \left( \nPi[m]_{\ell+1} \oplus \nPi[m]_{\ell-1} \right) \otimes \nPi[m]_{m-2}$.
	When $\ell = k$ we have that
	\[
	\nPi[m]_{\ell-1} \otimes \nPi[m]_{m-2} \cong \nPi[m]_{k} = \nPi[m]_{\ell}, \quad (\text{from } \eqref{eq: z/2 fusion rule})
	\]
	otherwise when $\ell = k-1$ we have that
	\[
	\nPi[m]_{\ell+1} \otimes \nPi[m]_{m-2} \cong \nPi[m]_{k-1} = \nPi[m]_{\ell}. \quad (\text{from } \eqref{eq: z/2 fusion rule})
	\]
	In both cases we have $\nPi[m]_{\ell} \overset{\oplus}{\subseteq} \nPi[m]_{\ell} \otimes \nPi[m]_{m-3}$ and so $\ufld(\alpha)$ contains an arrow $(\nPi[m]_{\ell},i) \ra (\nPi[m]_{\ell},j)$.
	On the other hand, since $v(\beta) = 3$, $\ufld(\beta)$ also contains an arrow $(\nPi[m]_{\ell},i) \ra (\nPi[m]_{\ell},j)$, which shows that 
	\[
	(\nPi[m]_{\ell},i) \rightrightarrows (\nPi[m]_{\ell},j)
	\]
	is a subquiver of $\check{Q}$.
	This shows that $\check{Q}$ is not of finite type.
\item (even case) $m = 2k$ for $k \geq 2$.\\
	As before, one shows that 
	\[
	\nPi[m]_{k-2} \overset{\oplus}{\subseteq} \nPi[m]_{k-1} \otimes \nPi[m]_{m-3}, \text{ and} \quad \nPi[m]_{k-1} \overset{\oplus}{\subseteq} \nPi[m]_{m-3} \otimes \nPi[m]_{k-2}.
	\]
	Together with the arrows from $\ufld(\beta)$, we get that the following is a subquiver of $\check{Q}$:
	\[\begin{tikzcd}[column sep = large]
	{(\nPi[m]_{k-1},i)} & {(\nPi[m]_{k-1},j)} \\
	{(\nPi[m]_{k-2},i)} & {(\nPi[m]_{k-2},j)}
	\arrow[from=1-1, to=1-2]
	\arrow[from=1-1, to=2-2]
	\arrow[from=2-1, to=1-2]
	\arrow[from=2-1, to=2-2]
	\end{tikzcd}\]
	which shows that $\check{Q}$ is not of finite type.
\end{itemize}
This covers all possible cases and hence completes the proof.
\end{proof}

\begin{figure}[ht]
\begin{center}
\begin{tabular}{||c | c ||} 
 \hline & \\ [-2ex]
 $\overline{Q}$ & $\overline{\check{Q}}$ \\ [1ex] 
 \hline\hline
 $B_n = C_n = \dynkin [Coxeter]B{}$ & $A_{2n-1} \coprod D_{n+1} = 
 \adjustbox{scale=0.8}{
	\begin{tikzcd}[every arrow/.append style = {shorten <= -.5em, shorten >= -.5em}]
	\bullet & \bullet & \bullet & \cdots & \bullet & \bullet & \bullet \\
	\bullet & \bullet & \bullet & \cdots & \bullet & \bullet & \bullet \\
	\bullet & \bullet & \bullet & \cdots & \bullet & \bullet & \bullet
	\arrow[no head, from=1-5, to=1-6]
	\arrow[no head, from=2-5, to=2-6]
	\arrow[no head, from=3-5, to=3-6]
	\arrow[no head, from=3-6, to=2-7]
	\arrow[no head, from=2-7, to=1-6]
	\arrow[no head, from=2-6, to=1-7]
	\arrow[no head, from=2-6, to=3-7]
	\arrow[no head, from=1-5, to=1-4, shorten <= -.5em, shorten >= .7em]
	\arrow[no head, from=2-5, to=2-4, shorten <= -.5em, shorten >= .7em]
	\arrow[no head, from=3-5, to=3-4, shorten <= -.5em, shorten >= .7em]
	\arrow[no head, from=1-2, to=1-3]
	\arrow[no head, from=1-4, to=1-3, shorten <= .7em, shorten >= -.5em]
	\arrow[no head, from=3-2, to=3-3]
	\arrow[no head, from=2-2, to=2-3]
	\arrow[no head, from=2-4, to=2-3, shorten <= .7em, shorten >= -.5em]
	\arrow[no head, from=3-4, to=3-3, shorten <= .7em, shorten >= -.5em]
	\arrow[no head, from=1-2, to=1-1]
	\arrow[no head, from=2-2, to=2-1]
	\arrow[no head, from=3-2, to=3-1]
	\end{tikzcd}
 }$ \\ 
 \hline
 $F_4 = \dynkin [Coxeter]F4{}$ & $E_6 \coprod E_6 =
 \adjustbox{scale=0.8}{
\begin{tikzcd}[every arrow/.append style = {shorten <= -.5em, shorten >= -.5em}]
	\bullet & \bullet & \bullet & \bullet \\
	\bullet & \bullet & \bullet & \bullet \\
	\bullet & \bullet & \bullet & \bullet
	\arrow[no head, from=1-1, to=1-2]
	\arrow[no head, from=2-1, to=2-2]
	\arrow[no head, from=3-1, to=3-2]
	\arrow[no head, from=3-2, to=2-3]
	\arrow[no head, from=2-3, to=1-2]
	\arrow[no head, from=2-2, to=1-3]
	\arrow[no head, from=1-3, to=1-4]
	\arrow[no head, from=2-2, to=3-3]
	\arrow[no head, from=3-3, to=3-4]
	\arrow[no head, from=2-3, to=2-4]
\end{tikzcd}
 }$ \\
 \hline
 $G_2 = \dynkin [Coxeter,gonality=6]G2$  & $A_5 \coprod A_5 =
 \adjustbox{scale=0.8}{
\begin{tikzcd}[every arrow/.append style = {shorten <= -.5em, shorten >= -.5em}]
	\bullet & \bullet & \bullet & \bullet & \bullet \\
	\bullet & \bullet & \bullet & \bullet & \bullet
	\arrow[no head, from=1-1, to=2-2]
	\arrow[no head, from=2-1, to=1-2]
	\arrow[no head, from=1-2, to=2-3]
	\arrow[no head, from=2-2, to=1-3]
	\arrow[no head, from=1-3, to=2-4]
	\arrow[no head, from=2-3, to=1-4]
	\arrow[no head, from=1-4, to=2-5]
	\arrow[no head, from=2-4, to=1-5]
\end{tikzcd}
 }$ \\
 \hline
 $H_3 = \dynkin [Coxeter]H3$ & $D_6 =
 \adjustbox{scale=0.8}{
\begin{tikzcd}[row sep = tiny, every arrow/.append style = {shorten <= -.5em, shorten >= -.5em}]
	& \bullet & \bullet \\
	\bullet \\
	& \bullet & \bullet \\
	\bullet
	\arrow[no head, from=4-1, to=3-2]
	\arrow[no head, from=3-2, to=2-1]
	\arrow[no head, from=2-1, to=1-2]
	\arrow[no head, from=1-2, to=1-3]
	\arrow[no head, from=3-2, to=3-3]
\end{tikzcd} 
 }$ \\
 \hline
 $H_4 = \dynkin [Coxeter]H4$ & $E_8 =
 \adjustbox{scale=0.8}{
\begin{tikzcd}[row sep = tiny, every arrow/.append style = {shorten <= -.5em, shorten >= -.5em}]
	& \bullet & \bullet & \bullet \\
	\bullet \\
	& \bullet & \bullet & \bullet \\
	\bullet
	\arrow[no head, from=4-1, to=3-2]
	\arrow[no head, from=3-2, to=2-1]
	\arrow[no head, from=2-1, to=1-2]
	\arrow[no head, from=1-2, to=1-3]
	\arrow[no head, from=3-2, to=3-3]
	\arrow[no head, from=1-3, to=1-4]
	\arrow[no head, from=3-3, to=3-4]
\end{tikzcd} 
 }$ \\ 
 \hline
 $I_2(2m) = \begin{tikzcd}[every arrow/.append style = {shorten <= -.5em, shorten >= -.5em}]
	\bullet & \bullet
	\arrow["2m", no head, from=1-1, to=1-2]
	\end{tikzcd}$ & $A_{2m-1} \coprod A_{2m-1} =
 \adjustbox{scale=0.8}{
\begin{tikzcd}[row sep=tiny, every arrow/.append style = {shorten <= -.5em, shorten >= -.5em}]
	\bullet & \bullet & \bullet & {{}} & \bullet \\
	&&& \cdots \\
	\bullet & \bullet & \bullet & {{}} & \bullet
	\arrow[no head, from=1-1, to=3-2]
	\arrow[no head, from=3-1, to=1-2]
	\arrow[no head, from=1-2, to=3-3]
	\arrow[no head, from=3-2, to=1-3]
	\arrow[no head, from=3-3, to=1-4, shorten >= .6em]
	\arrow[no head, from=1-3, to=3-4, shorten >= .6em]
	\arrow[no head, from=3-4, to=1-5, shorten <= .6em]
	\arrow[no head, from=1-4, to=3-5, shorten <= .6em]
\end{tikzcd}
 }$ \\ 
 \hline
 $I_2(2m+1) = \begin{tikzcd}[every arrow/.append style = {shorten <= -.5em, shorten >= -.5em}]
	\bullet & \bullet
	\arrow["2m+1", no head, from=1-1, to=1-2]
	\end{tikzcd}$ & $A_{2m} =
 \adjustbox{scale=0.8}{
\begin{tikzcd}[row sep=tiny, every arrow/.append style = {shorten <= -.5em, shorten >= -.5em}]
	\bullet &  & \bullet & {{}} & \\
	&&& \cdots \\
	 & \bullet &  & {{}} & \bullet
	\arrow[no head, from=1-1, to=3-2]
	\arrow[no head, from=3-2, to=1-3]
	\arrow[no head, from=1-3, to=3-4, shorten >= .6em]
	\arrow[no head, from=1-4, to=3-5, shorten <= .6em]
\end{tikzcd}
 }$ \\ [1ex] 
\hline
\end{tabular}
\end{center}
\caption{Unfolding of finite type non-simply-laced Coxeter graphs to finite type simply-laced Coxeter graphs.}
\label{fig: unfolding finite type coxeter}
\end{figure}

The underlying labelled graph of a Coxeter quiver $Q$ is defined to be the labelled graph obtained from $Q$ by forgetting the orientation.
A Coxeter quiver is said to be \emph{connected} if its underlying graph is connected.
We are now ready to prove Gabriel's classification theorem for Coxeter quivers.
\begin{theorem}\label{thm: gabriel classification}
Let $Q$ be a connected Coxeter quiver.
Then $Q$ is of finite type if and only if the underlying labelled graph of $Q$ is a Coxeter--Dynkin diagram: \\
\(A_n = \dynkin [Coxeter]A{} \), \
\(B_n = C_n = \dynkin [Coxeter]B{}\), \
\(D_n = \dynkin [Coxeter]D{} \), \
\(E_6 = \dynkin [Coxeter]E6\), \newline
\(E_7 = \dynkin [Coxeter]E7\), \
\(E_8 = \dynkin [Coxeter]E8 \), \
\(F_4 = \dynkin [Coxeter]F4\), \
\(G_2 = \dynkin [Coxeter,gonality=6]G2 \), \
\(H_3 = \dynkin [Coxeter]H3\), \vspace{15pt} \\ 
\(H_4 = \dynkin [Coxeter]H4\), \
\(I_2(n) = \dynkin [Coxeter,gonality=n]I{}\).
\end{theorem}
\begin{proof}
Let $\check{Q}$ be the unfolded quiver of $Q$.
We denote the underlying labelled graph of $Q$ by $\overline{Q}$ and the underlying (unlabelled) graph of $\check{Q}$ by $\overline{\check{Q}}$.

If $\overline{Q}$ is a Coxeter--Dynkin diagram, then one can easily check that the connected components of $\overline{\check{Q}}$ are $ADE$ Dynkin diagrams; refer to \cref{fig: unfolding finite type coxeter} for the complete list.
By \cref{thm: rep Q and checkQ equiv} and the classical Gabriel's theorem, it follows that $Q$ is of finite type.

Suppose instead that $Q$ is of finite type. 
By \cref{thm: rep Q and checkQ equiv} its unfolded quiver $\check{Q}$ has only finitely many indecomposable representations, hence the classical Gabriel's theorem says that $\check{Q}$ is a disjoint union of $ADE$ Dynkin diagrams.
In particular, the classification of finite Coxeter groups says that the associated Coxeter group $W\left(\overline{\check{Q}}\right)$ is a finite group.
Since $Q$ is of finite type by assumption, \cref{lem: multilaced quiver not finite type} implies that $Q$ can not have more than one labelled arrow between any pair of vertices.
In particular, its underlying labelled graph $\overline{Q}$ is actually a Coxeter graph.
By construction, the map $f: \overline{\check{Q}}_0 \ra \overline{Q}_0$ sending $(A,i) \xmapsto{f} i$ defines an (admissible) folding map of Coxeter graphs (in the sense of \cite[Definition 4.1]{Crisp99} or \cite[Section 3.1]{Lusz83}).
As such, it induces an injective Coxeter group homomorphism $W(\overline{Q}) \ra W\left(\overline{\check{Q}}\right)$; see \cite[Proposition 2.3 and 4.3]{Crisp99} or \cite[Corollary 3.3]{Lusz83}.
Since $W\left(\overline{\check{Q}}\right)$ is finite, it follows that $W(\overline{Q})$ must be finite. 
By the classification of finite Coxeter groups we get that $\overline{Q}$ must be a Coxeter--Dynkin diagram.
\end{proof}

\section{Reflection functors and positive roots} \label{sec: reflection and roots}
In this section, we use reflection functors similar to those defined in \cite{BGP_73, DR_76} to relate indecomposable representations to positive roots of the underlying Coxeter graph.
The root system at play will be the ones studied in \cite{Dyer09}, which are defined over (commutative) fusion rings instead of $\Z$.
We start by briefly recalling this generalised notion of root systems.

\subsection{Root system over fusion rings and dimension vectors} \label{subsec: root system in fusion ring}
Note that if $Q$ and $Q'$ are Coxeter quivers with the same underlying labelled graph $\Gamma:= \overline{Q} = \overline{Q'}$, then we must have that $\cC(Q) = \cC(Q')$.
To emphasise this fact, we shall also use the notation 
\[
\cC(\Gamma) := \cC(Q) = \cC(Q').
\]
Let $\Gamma$ be the underlying labelled graph of a Coxeter quiver, with vertex set $\Gamma_0$, edge set $\Gamma_1$ and label map $v: \Gamma_1 \ra \{n \in \mathbb{N} \mid n \geq 3\}$.
Define $L_{\Gamma}$ to be the free $K_0(\cC(\Gamma))$-module
\[
L_\Gamma := \bigoplus_{i \in \Gamma_0} K_0(\cC(\Gamma)).
\]
Elements in $L_\Gamma$ shall be represented as $(v_i)_{i \in \Gamma_0}$, where we think of them as a vector with entries in $K_0(\cC(\Gamma))$.
The standard basis elements with $[\1_{\cC(Q)}] \in K_0(\cC(\Gamma))$ at position $i$ and 0 elsewhere will be denoted by $e_i$.
The following is a slight generalisation of root systems over commutative rings for labelled graphs.
\begin{definition} \label{defn: root system}
Let $\Gamma:= \overline{Q}$ be the underlying labelled graph of a Coxeter quiver $Q$.
\begin{enumerate}[(i)]
\item An element $x$ of the (commutative) fusion ring $K_0(\cC(\Gamma))$ is called \emph{positive} if  $x = [X]$ for some non-zero object $X \in \cC(Q)$. 
The set of positive elements is denoted by $K_0(\cC(\Gamma))_{>0}$.
Similarly we say that $v:= (v_i)_{i \in \Gamma_0} \in L_\Gamma$ is \emph{positive} if each $v_i$ is positive or zero, but not all zero (namely $v \neq 0$). 
Both notions of positivity are clearly closed under addition. \label{defn: positive}
\item Define $B: L_\Gamma \times L_\Gamma \ra K_0(\cC(\Gamma))$ to be the (symmetric) $K_0(\cC(\Gamma))$-bilinear form uniquely determined by
\[
B(e_i,e_j) := \begin{cases}
	2[\1_{\cC(\Gamma)}], &\text{ when } i = j; \\
	- \sum_{\substack{\alpha \in \Gamma_1 \\ \alpha \text{ connects $i$ and $j$} } } \left[\nPi[v(\alpha)]_{v(\alpha)-3}\right], &\text{ when } i\neq j.
	\end{cases}
\]
\item For each $i \in \Gamma_0$, we define the $K_0(\cC(\Gamma))$-linear map $\sigma_i: L_\Gamma \ra L_\Gamma$ by
\[
\sigma_i(v) := v - B(e_i,v)e_i.
\]
It is easy to see that $\sigma_i$ is an involution.\label{defn: reflection group}
The group of $K_0(\cC(\Gamma))$-linear maps generated by $\sigma_i$ for $i \in \Gamma_0$ is denoted by $W_\Gamma$. 
\item The elements $e_i \in L_\Gamma$ are called the \emph{simple roots} and elements $v \in L_\Gamma$ living in the $W_\Gamma$-orbit of the simple roots are called \emph{roots}.
An \emph{extended root} is an element in $L_\Gamma$ given by $[A]\cdot v$ for $v$ a root and $A$ a simple object in $\cC(Q)$.
A root (similarly extended root) is \emph{positive} if it is positive as an element in $L_\Gamma$.
\item The set of extended roots is denoted by $\Phi^{ext}$ and the set of positive extended roots is denoted by $\Phi_+^{ext}$.
On the other hand, the set of roots is denoted by $\Phi \subseteq \Phi^{ext}$ and the set of positive roots is denote by $\Phi_+ \subseteq \Phi_+^{ext}$.
It is easy to check from the definition that
\[
\Phi^{ext} = \coprod_{A \in \Irr(\cC(\Gamma))} [A]\cdot \Phi, \quad \Phi_+^{ext} = \coprod_{A \in \Irr(\cC(\Gamma))} [A]\cdot \Phi_+. 
\]
\end{enumerate}
\end{definition}

When $\Gamma$ is a Coxeter graph, i.e. no two vertices in $\Gamma$ are connected by two distinct edges, the datum
\[
\left(K_0(\cC(\Gamma)), K_0(\cC(\Gamma))_{>0}, L_\Gamma, L_\Gamma, B(-,-), \{e_i\}_{i \in \Gamma_0}, \{e_i\}_{i \in \Gamma_0}, \id \right)
\]
is a root basis datum with positivity structure in the sense of \cite[Section 2.3]{Dyer09} (see also Section 3.4 and 3.5 in loc. cit.).
In particular, we have the following result:
\begin{proposition}[\protect{\cite[Proposition 2.4]{Dyer09}}] \label{prop:rootproperties}
Suppose $\Gamma = \overline{Q}$ is a Coxeter graph.
Then we have the following:
\begin{enumerate}[(i)]
\item $W_\Gamma$ is isomorphic to the Coxeter group associated to $\Gamma$.
\item Every root and extended root is either positive or negative; namely
\[
\Phi = \Phi_+ \coprod -\Phi_+, \quad \Phi^{ext} = \Phi^{ext}_+ \coprod -\Phi^{ext}_+ .
\]
\end{enumerate}
\end{proposition}

\begin{definition}\label{defn: dim vector}
Let $V$ be a representation of a Coxeter quiver $Q$.
We define the \emph{dimension vector} $\dim(V)$ of $V$ as
\[
\dim(V) := ([V_i])_{i \in Q_0} \in L_{\overline{Q}}.
\]
\end{definition}
Note that if $Q$ is a classical quiver ($\img(v) = 3$), $K_0(\cC(\overline{Q})) \cong \Z$ and $\dim(V)$ for $V \in \rep(Q)$ corresponds exactly to the classical dimension vector of $V$.

\begin{remark}
Let $S(i)$ be the simple representation of $Q$ such that $S(i)_i = \1_{\cC(Q)}$ and is zero everywhere else.
It is easy to see that the Grothendieck classes $\{ [S(i)] \mid i \in Q_0 \} \subset K_0(\rep(Q))$ form a $K_0(\cC(\overline{Q}))$-basis for $K_0(\rep(Q))$, so that $K_0(\rep(Q)) \cong L_{\overline{Q}}$.
\end{remark}

\subsection{Reflection functors and Gabriel's theorem on indecomposables and positive roots}
Throughout, $Q$ will be a Coxeter quiver and $\check{Q}$ will be its associated unfolded classical quiver.
We shall reserve $S(i)$ to denote the simple representation of $Q$ such that $S(i)_i = \1_{\cC(Q)}$ and is zero everywhere else.

Before we begin, let us set up some notation.
Recall that every object in $\cC(Q)$ is self-dual, so we have the following natural isomorphism between morphism spaces in $\cC(Q)$:
\[
\Hom_{\cC(Q)}(A \otimes B_1, B_2) \cong \Hom_{\cC(Q)}(B_1, A \otimes B_2).
\]
More precisely, the isomorphism can be described as follows.
For each $f: A\otimes B_1 \ra B_2$, we denote the image of $f$ in $\Hom_{\cC(Q)}(B_1, A \otimes B_2)$ induced by the duality of $A$ by $\hat{f}$, which diagrammatically is given by:
\begin{equation} \label{eqn: induced dual map}
\begin{tikzpicture}
\node (1A) {};
\node[right=1.5ex of 1A] (1B2) {$B_2$};
\node[below=11ex of 1A] (2A) {$A$};
\node[right=4.5ex of 2A] (2B1) {$B_1$};
\node[below=3ex of 1B2,draw] (f) {$f$};

\draw (1B2) -- (f);
\draw (2A) -- (f);
\draw (2B1) -- (f);
\node[right=6ex of f] (squig) {$\mapsto$};
\node[right=3ex of squig] (fhat) {$\hat{f} :=$};
\node[right=1ex of fhat] (2A'1) {$A$};
\node[right=2.5ex of 2A'1] (2A'2) {$A$};
\node[right=4.5ex of 2A'2] (2B'1) {$B_1$};
\node[above=8ex of 2A'1] (1A') {$A$};
\node[right=6ex of 1A'] (1B'2) {$B_2$};
\node[below=6ex of 2B'1] (3B'1) {$B_1,$};     
\node[below=2ex of 1B'2,draw] (f') {$f$};

\draw (2A'1.south) arc[start angle=-180, end angle=0, radius=3ex] (2A'2);
\draw (1A') -- (2A'1);
\draw (2B'1) -- (3B'1);
\draw (1B'2) -- (f');
\draw (2A'2) -- (f');
\draw (2B'1) -- (f');
\end{tikzpicture}
\end{equation}
where the cup $\cup: \1_{\cC(Q)} \ra A \otimes A$ denotes the unit morphism coming from the self-duality of $A$  ($\1_{\cC(Q)}$ omitted from the diagram since $\cC(Q)$ is strictly monoidal) and the straight lines denote the corresponding identity morphisms.

Given any vertex $i \in Q_0$, we define $\sigma_i Q$ to be the Coxeter quiver that reverses all the arrows $\alpha$ that start or end at $i$, whilst preserving the labels $v(\alpha)$.
Following the classical notion we say that a vertex is a \emph{sink} (resp.\ \emph{source}) if it only has incoming (resp.\ outgoing) arrows.

The following fact (similar to the classical case) can be easily checked using the semi-simple property of $\cC(Q)$.
\begin{lemma} \label{lem: simple epi mono condition}
Let $V$ be an indecomposable representation of $Q$.
\begin{enumerate}[(i)]
\item Let $i \in Q_0$ be a sink.
Then either $V \cong S(i) \otimes A$ for some $A \in \Irr(\cC(Q))$, or the following morphism $\xi$ in $\cC(Q)$
\[
\xi:= \left[ V_{\alpha} \right]^T_{\substack{ \scriptscriptstyle
	\alpha \in Q_1 \\ \scriptscriptstyle
	t(\alpha) = i}
	}
		: \bigoplus_{\substack{
			\alpha \in Q_1 \\
			t(\alpha) = i}
			}	 
				\nPi[v(\alpha)]_{v(\alpha)-3} \otimes V_{s(\alpha)} \ra V_i
\]
is an epimorphism.
\item Let $i \in Q_0$ be a source.
Then either $V \cong S(i) \otimes A$ for some $A \in \Irr(\cC(Q))$, or the following morphism $\vartheta$ in $\cC(Q)$
\[
\vartheta:= \left[ \hat{V}_{\alpha} \right]_{\substack{ \scriptscriptstyle
	\alpha \in Q_1 \\ \scriptscriptstyle
	s(\alpha) = i}
	}
		: V_i \ra \bigoplus_{\substack{
		\alpha \in Q_1 \\
		s(\alpha) = i}
		} 
			\nPi[v(\alpha)]_{v(\alpha)-3} \otimes V_{t(\alpha)}
\]
is a monomorphism (refer to \eqref{eqn: induced dual map} for the notation $\hat{V}_\alpha$).
\end{enumerate}
\end{lemma}

\begin{definition}[Reflection functors]
Let $i \in Q_0$ be a sink.
We define $R^+_i: \rep(Q) \ra \rep(\sigma_iQ)$ to be the functor constructed as follows.
Let the morphism
\[
\xi : \bigoplus_{\substack{
			\alpha \in Q_1 \\
			t(\alpha) = i}
			}	 
				\nPi[v(\alpha)]_{v(\alpha)-3} \otimes V_{s(\alpha)} \ra V_i
\]
be as defined in \cref{lem: simple epi mono condition}.
Given a representation $V \in \rep(Q)$, $R^+_i(V)$ is the representation of $\sigma_i Q$ such that for all vertices $j\neq i \in Q_0$, $R^+_i(V)_j = V_j$.
On the vertex $i$ we set $R^+_i(V)_i := \ker(\xi)$ to be the kernel of $\xi$, fitting into the exact sequence
\[
0 \ra \ker(\xi) \xra{\frk(\xi)} \bigoplus_{\substack{
		\alpha \in Q_1 \\
		t(\alpha) = i}
		} 
			\nPi[v(\alpha)]_{v(\alpha)-3} \otimes V_{s(\alpha)} \ra V_i.
\]
For each arrow $\alpha$ not ending at $i$, we set $R^+_i(V)_{\alpha} = V_{\alpha}$.
For each arrow $\beta$ ending at $i$, we consider the morphism $\Psi_\beta$ defined as the projection of $\frk(\xi)$ onto the summand $\nPi[v(\beta)]_{v(\beta)-3} \otimes  V_{s(\beta)}$:
\[
\Psi_\beta: \ker(\xi) \xhookrightarrow{\frk(\xi)} \bigoplus_{\alpha \in A_i} \nPi[v(\alpha)]_{v(\alpha)-3} \otimes V_{s(\alpha)} \xtwoheadrightarrow{\rho} \nPi[v(\beta)]_{v(\beta)-3} \otimes  V_{s(\beta)},
\]
and define $R^+_i(V)_\beta := \hat{\Psi}_\beta$ (refer to \eqref{eqn: induced dual map} for the notation $\hat{\Psi}_\beta$). 

\noindent
Given a morphism $f: V \ra W$ of representations of $Q$, we define $R^+_i(f): R^+_i(V) \ra R^+_i(W)$ by setting $R^+_i(f)_j= f_j$ for all $j\neq i$ and $R^+_i(f)_i$ is the unique morphism induced by the kernel property defining $R^+_i(V)$ and $R^+_i(W)$:
\[\begin{tikzcd}[row sep = large]
	0 & {R^+_i(V)_i} & {\bigoplus_{\substack{
			\alpha \in Q_1 \\
			t(\alpha) = i}
			}	 
				\nPi[v(\alpha)]_{v(\alpha)-3} \otimes V_{s(\alpha)}} & {V_i} \\
	0 & {R^+_i(W)_i} & {\bigoplus_{\substack{
			\alpha \in Q_1 \\
			t(\alpha) = i}
			}	 
				\nPi[v(\alpha)]_{v(\alpha)-3} \otimes W_{s(\alpha)}} & {W_i}
	\arrow[from=1-1, to=1-2]
	\arrow[from=1-2, to=1-3]
	\arrow[from=1-3, to=1-4]
	\arrow[from=2-1, to=2-2]
	\arrow[from=2-2, to=2-3]
	\arrow[from=2-3, to=2-4]
	\arrow["{\exists ! R^+_i(f_i)}"', dashed, from=1-2, to=2-2]
	\arrow["{\bigoplus_{\substack{ \scriptscriptstyle
			\alpha \in Q_1 \\ \scriptscriptstyle
			t(\alpha) = i}
			}	 
				\id \otimes f_{s(\alpha)}}"', from=1-3, to=2-3]
	\arrow["{f_i}"', from=1-4, to=2-4]
\end{tikzcd}\]
It is easy to check that $R^+_i(f)$ satisfies the required property to be a morphism of representations.
If $i \in Q_0$ is instead a source, one can define $R^-_i : \rep(Q) \ra \rep(\sigma_iQ)$ in a similar fashion, using instead the cokernel of the morphism $\vartheta$ defined in \cref{lem: simple epi mono condition}.
\end{definition}
The following property of reflection functors follows from the kernel--cokernel properties:
\begin{lemma} \label{lem: reflection almost mutual inverse}
Let $V \in \rep(Q)$.
\begin{enumerate}[(i)]
\item Suppose $i$ is a sink and $\xi$ is as in \cref{lem: simple epi mono condition}. 
If $\xi$ is an epimorphism, then $R^-_i R^+_i (V) \cong V$.
\item Suppose $i$ is a source and $\vartheta$ is as in \cref{lem: simple epi mono condition}.
If $\vartheta$ a monomorphism, then $R^+_i R^-_i (V) \cong V$.
\end{enumerate}
\end{lemma}

Note that $\Gamma:= \overline{Q} = \overline{\sigma_iQ}$ for any $i$, so that the dimension vectors $\dim(V)$ and $\dim(V')$ with $V \in \rep(Q)$ and $V' \in \rep(\sigma_iQ)$ are both elements in $L_\Gamma = \bigoplus_{i \in \Gamma_0} K_0(\cC(\Gamma))$.
Similar to the classical case, reflection functors and reflections on $L_\Gamma$ are related through dimension vectors as follows.
\begin{proposition} \label{prop: reflection functor and reflection element}
Let $V$ be an indecomposable representation of $Q$ and let $i \in Q_0$ be a sink (resp. source). 
If $V \not\cong S(i) \otimes A$ for any $A \in \Irr(\cC(Q))$, then 
\[
\dim(R^\pm_i(V)) = \sigma_i(\dim(V)) \in L_{\overline{Q}}.
\]
\end{proposition}
\begin{proof}
We prove the case where $i$ is a sink; the proof for when $i$ is a source follows similarly.

By \cref{lem: simple epi mono condition}, the assumption on $V$ gives us the follows exact sequence
\[
0 \ra \ker(\xi) \xra{\frk(\xi)} \bigoplus_{\substack{
		\alpha \in Q_1 \\
		t(\alpha) = i}
		} 
			\nPi[v(\alpha)]_{v(\alpha)-3} \otimes V_{s(\alpha)} \ra V_i \ra 0,
\]
where $\ker(\xi) = R^+_i(V)_i$ by definition.
This implies that
\begin{align*}
[R^+_i(V)_i] &= \left[\bigoplus_{\substack{
		\alpha \in Q_1 \\
		t(\alpha) = i}
		} 
			\nPi[v(\alpha)]_{v(\alpha)-3} \otimes V_{s(\alpha)} \right] - [V_i] \\
&= \left( \sum_{\substack{
		\alpha \in Q_1 \\
		t(\alpha) = i}
		} 
			\left[\nPi[v(\alpha)]_{v(\alpha)-3}\right][V_{s(\alpha)}] \right) - [V_i] \\
&= [V_i] - B(e_i, \dim(V)).
\end{align*}
Since $R^+_i(V)_j = V_j$ for all $j \neq i$, it follows that $\dim(R^+_i(V)) = \sigma_i(\dim(V))$ as required.
\end{proof}

From now on we shall focus on the case where $\Gamma := \overline{Q}$ is a Coxeter--Dynkin diagram (refer to \cref{thm: gabriel classification} for the classification).
Let $n := |Q_0|$ be the number of vertices of $Q$.
In this case, the vertices $Q_0$ of the Coxeter quiver $Q$ can always be equipped with an \emph{admissible (sink) ordering}: the vertices of $Q_0$ are labelled by $1,2,...,n$ so that $1$ is a sink in $Q$ and for each $i \in \{2,...,n\}$, $i$ is a sink in the quiver $\sigma_{i-1}\cdots \sigma_2\sigma_1Q$ and moreover $\sigma_n \cdots \sigma_1Q = Q$.
\begin{definition}
Let $\Gamma := \overline{Q}$ be a Coxeter--Dynkin diagram and suppose $Q$ is equipped with an admissible (sink) ordering.
The element $c:= \sigma_n \cdots \sigma_2\sigma_1 \in W_\Gamma$ (see \cref{defn: root system}(\ref{defn: reflection group}) for the definition of $W_\Gamma$) is called the \emph{Coxeter element} associated to the ordering.
Similarly, the composition of reflection functors $C^+:= R^+_n \cdots R^+_1 : \rep(Q) \ra \rep(Q)$ is called the \emph{Coxeter functor} associated to the ordering.
\end{definition}
The following slight generalisation of a (well-known) lemma will be useful to us.
\begin{lemma} \label{lem: coxeter eventually non-positive}
Suppose $\Gamma = \overline{Q}$ is a Coxeter--Dynkin diagram with an admissible (sink) ordering on $Q$.
If $v := (v_i)_{i \in \Gamma_0} \in L_\Gamma$ is a positive element (see \cref{defn: root system}(\ref{defn: positive})), then there exists $r \in \mathbb{N}$ such that $c^r(v)$ is no longer a positive element.
\end{lemma}
\begin{proof}
When $\Gamma$ is a Coxeter--Dynkin diagram, it is easy to check that $B(-,-)$ is non-degenerate; namely $B(-,v) = 0$ if and only if $v = 0$.
It follows that $c(x) = x \iff x=0$.
Moreover, $W_\Gamma$ is a finite Coxeter group by \cref{prop:rootproperties}, so $c^m = \id$ for some $m \in \mathbb{N}$.
This means that with $x:= v + c(v) + \cdots + c^{m-1}(v)$, we get $c(x) = x$, which implies that $x = 0$.
Since $v$ is a positive element and the set of positive elements is closed under addition, it must be the case that at least one of $c(v), \cdots, c^{m-1}(v)$ is not a positive element.
\end{proof}

We are now ready to prove the second part of Gabriel's theorem for Coxeter quivers.
\begin{theorem}\label{thm: gabriel root}
Suppose $Q$ is a Coxeter quiver whose underlying labelled graph $\Gamma:= \overline{Q}$ is a Coxeter--Dynkin diagram.
The dimension vector map $\dim$ induces a bijection between the set of indecomposable representations of $Q$ (up to isomorphism) and the set of positive extended roots $\Phi_+^{ext}$ of $L_\Gamma$.
Moreover, every indecomposable representation $V \in \rep(Q)$ is of the form
\[
V \cong V' \otimes A
\]
for some $V' \in \rep(Q)$ that corresponds to a positive root (i.e. $\dim(V') \in \Phi_+$) and some simple object $A$ in $\cC(\Gamma)$.
\end{theorem}
\begin{proof}
The proof of the first statement is essentially identical to the classical case, using analogous lemmas that we have prepared beforehand.
Start by fixing some admissible (sink) ordering on $Q_0$.
The fact that dimension vectors can not be negative together with \cref{prop: reflection functor and reflection element} and \cref{lem: coxeter eventually non-positive} imply that  for any indecomposable $V \in \rep(Q)$, there is a shortest possible expression $R^+_{k+1} \cdots R^+_1 (C^+)^r$ such that
\begin{equation}\label{eqn: coming from simples}
R^+_k \cdots R^+_1 (C^+)^r (V) \cong S(k+1) \otimes A
\end{equation}
for some $A \in \Irr(\cC(\Gamma))$.
Moreover, using \cref{prop: reflection functor and reflection element} again, we get 
\[
\dim(V) = c^{-r}s_1\cdots s_k ([A] \cdot e_{k+1}) = [A]\cdot (c^{-r}s_1\cdots s_k (e_{k+1})) \in \Phi_+^{ext},
\]
which shows that $\dim$ indeed maps indecomposables to positive extended roots.

Now suppose $\dim(V) = \dim(V')$ for two indecomposables $V, V'\in \rep(Q)$.
Since $\dim(V) = \dim(V')$, the same functor $R^+_k \cdots R^+_1 (C^+)^r$ must send both $V$ and $V'$ to simple objects $S(k+1)\otimes A$ and $S(k+1)\otimes A'$ respectively, where  $A$ and $A'$ are a priori possibly distinct simple objects in $\cC(\Gamma)$.
However, since
\[
[A]\cdot e_{k+1} = s_k \cdots s_1 c^r (\dim(V)) = s_k \cdots s_1 c^r (\dim(V')) = [A']\cdot e_{k+1},
\]
being free over $K_0(\cC(\Gamma))$ says that $[A] = [A']$ and hence $A\cong A'$.
It follows that $V \cong V'$ by applying the appropriate (inverse) reflection functors.
This shows that $\dim$ is injective.

Now to show surjectivity, suppose $v \in \Phi^{ext}_+$ is an arbitrary positive extended root.
Then once again we have some shortest possible expression $\tau:= s_{k+1} \cdots s_1 c^r$ such that $\tau(v)$ is no longer positive.
Since $\Phi^{ext} = \Phi^{ext}_+ \coprod -\Phi^{ext}_+$ by \cref{prop:rootproperties} and each $s_i$ only changes the coefficient of one basis element, it fact we must have that $s_k\cdots s_1 c^r(v) = [A] \cdot e_{k+1}$ for some $A \in \Irr(\cC(\Gamma))$.
Consider the functor $C^- := R^-_1\cdots R^-_n$.
It is clear now that the indecomposable object $V := (C^-)^r R^-_1 \cdots R^-_k(S(k+1))\otimes A$ maps to $v$ under $\dim$, using the same arguments as before.

For the final statement, given $V$ together with \eqref{eqn: coming from simples}, we set $V':= (C^-)^r R^-_1 \cdots R^-_k(S(k+1))$, where it follows that $V' \otimes A \cong V$ with $\dim(V') \in \Phi_+$ by construction.
\end{proof}

\begin{example}
We continue again with the Coxeter quiver from \cref{eg:I2(5)Coxeterquiver}.
Its underlying labelled graph $\Gamma := \overline{Q}$ is of type $I_2(5) = \dynkin [Coxeter,gonality=5]I{}$, and so the corresponding group $W_\Gamma$ is isomorphic to the Coxeter group of $I_2(5)$, which is the group of isometries of the regular pentagon.

For simplicity, we shall identify $K_0(\cC(Q))$ with $\bbZ[\delta] \subset \bbR$ via sending $[\1] \mapsto 1$ and $[\Pi] \mapsto \delta$, where $\delta = 2\cos(\pi/5)$ is the golden ratio satisfying $\delta^2 = 1+\delta$.
The 10 extended positive roots in $L_\Gamma = \bbZ[\delta] \oplus \bbZ[\delta]$ are listed below.
The 5 positive roots are listed on the left column and the extended positive roots on the right are obtained from the left by scalar multiplying by $\delta$.
\begin{multicols}{2}
\begin{center}
\begin{enumerate}[(i)]
\item $\begin{pmatrix} 1 & 0 \end{pmatrix}$; \\
\item $\begin{pmatrix} 0 & 1 \end{pmatrix}$; \\
\item $\begin{pmatrix} 1 & \delta \end{pmatrix}$; \\
\item $\begin{pmatrix} \delta & 1 \end{pmatrix}$; \\
\item $\begin{pmatrix} \delta & \delta \end{pmatrix}$; \\
\item $\begin{pmatrix} \delta & 0 \end{pmatrix}$; \\
\item $\begin{pmatrix} 0 & \delta \end{pmatrix}$; \\
\item $\begin{pmatrix} \delta & 1+\delta \end{pmatrix}$; \\
\item $\begin{pmatrix} 1+\delta & \delta \end{pmatrix}$; \\
\item $\begin{pmatrix} 1+\delta & 1+\delta \end{pmatrix}$.
\end{enumerate}
\end{center}
\end{multicols}
This list is also in correspondence with the list of indecomposable objects in \cref{eg:unfoldedA4} via taking the dimension vector.
In particular, the indecomposable representations on the right column in \cref{eg:unfoldedA4} are obtained from the left by tensoring with $\Pi$ on the right.
\end{example}

\section{Path algebra and its modules}\label{sec: path algebra}
As in the classical case, we show that a suitable notion of path algebra of a Coxeter quiver $Q$ can be defined so that its category of modules is equivalent (as abelian \emph{module categories} over $\cC(Q)$) to the category of representations of $Q$.
We first recall the notions of an algebra and its modules in fusion categories; see e.g. \cite[Section 3.1]{ostrik_2003}.

\subsection{Algebras and modules in tensor categories}
Throughout this subsection, $\cC$ denotes a strict fusion category.

\begin{definition}
An \emph{algebra} in $\cC$ is an object $A$ in $\cC$ equipped with two maps $\mu: A\otimes A \ra A$ (multiplication) and $\eta: \1 \ra A$ (unit), satisfying the associativity and the left and right unital constraints:
\begin{center}
\begin{tikzcd}
A\otimes A \otimes A 
	\ar[r, "\mu \otimes \id"] 
	\ar[d, "\id \otimes \mu"] 
& A\otimes A 
	\ar[d, "\mu"] \\
A \otimes A 
	\ar[r, "\mu"] 
& A
\end{tikzcd}
(associativity), \quad
\begin{tikzcd}
\1 \otimes A 
	\ar[r, "\eta \otimes \id"] 
	\ar[dr, "=", swap] 
& A\otimes A  
	\ar[d, "\mu"]
& A \otimes \1 
	\ar[l, "\id \otimes \eta", swap] 
	\ar[dl, "="] \\
& A  
\end{tikzcd}
(unital).
\end{center}
\end{definition}

\begin{definition}\label{defn: modules over algebra}
Let $A$ and $B$ be algebras in $\cC$.
A left (resp. right) \emph{module} of $A$ is an object $M$ in $\cC$ equipped an action map $\lambda: A\otimes M \ra M$ (resp. $\lambda: M\otimes A \ra M$) satisfying the associativity and the left (resp. right) unital constraints:
\begin{center}
\begin{tikzcd}
A\otimes A \otimes M 
	\ar[r, "\mu \otimes \id"] 
	\ar[d, "\id \otimes \lambda"] 
& A\otimes M 
	\ar[d, "\lambda"] \\
A \otimes M 
	\ar[r, "\lambda"] 
& M
\end{tikzcd}
(associativity), \quad
\begin{tikzcd}
\1 \otimes M 
	\ar[r, "\eta \otimes \id"] 
	\ar[dr, "=", swap] 
& A\otimes M  
	\ar[d, "\lambda"] \\
& M  
\end{tikzcd}
(left unital).
\end{center}
A \emph{bimodule} over $(A,B)$ is an object $M$ which is both a left module over $A$ and a right module over $B$, where the two module structures are compatible:
\begin{center}
\begin{tikzcd}
A\otimes M \otimes B 
	\ar[r, "\lambda \otimes \id"] 
	\ar[d, "\id \otimes \lambda'", swap] 
& M\otimes B 
	\ar[d, "\lambda'"] \\
A \otimes M 
	\ar[r, "\lambda"] 
& M
\end{tikzcd} (compatible actions).
\end{center}
We denote the category of left (resp. right) modules over $A$ as ${}_A\mod$ (resp. $\mod_A$); similarly the category of bimodules over $(A,B)$ is denoted as ${}_A\mod_{B}$.
\end{definition}

Note that ${}_A \mod$ (resp. $\mod_A$) is naturally a right (resp. left) module category over $\cC$; namely there is a natural tensor functor from $\cC^{\otimes op}$ to $\cEnd({}_A\mod)$ defined by $A \mapsto -\otimes A$ (resp. $\cC^{\otimes op}$ to $\cEnd(\mod_A)$ defined by $A \mapsto A\otimes -$).

\begin{definition}
An algebra $A$ in $\cC$ is said to be \emph{semi-simple} if its category of left modules is semi-simple.
\end{definition}

\begin{example}
\begin{enumerate}[(i)]
\item Let $\1$ denote the monoidal unit of $\cC$. Then $\1$ is always an algebra object, induced by the monoidal structure of $\cC$. Moreover, ${}_{\1}\mod$ is equivalent to $\cC$ and therefore $\1$ is a semi-simple algebra. 
\item The direct sum $\bigoplus_i \1$ can also be made into a semi-simple algebra, where multiplication between differently-indexed $\1$'s are zero.
\end{enumerate}
\end{example}

\begin{definition}
Let $A, B$ and $C$ be algebra objects in $\cC$.
For $(M, \theta)$ and $(N, \theta')$ objects in ${}_A\mod_B$ and ${}_B\mod_C$ respectively, we define $M\otimes_B N$ to be the cokernel of the morphism $\theta\otimes \id - \id \otimes \theta'$
\[
M\otimes_B N := \coker(M\otimes B \otimes N \xra{\theta\otimes \id - \id \otimes \theta'} M \otimes N) \in {}_A\mod_C.
\]
\end{definition}

\subsection{Coxeter quiver path algebras} \label{subsec: path algebra}
Let $Q=(Q_0, Q_1, s,t,v)$ be a Coxeter quiver and let $\cC(Q)$ be its associated fusion category.
We shall construct an algebra object associated to $Q$ in $\cC(Q)$ as follows.

Firstly, recall that the monoidal unit $\1$ is  an algebra object in a canonical way.
For each vertex $i \in Q_0$, let $(e_i := \1, \mu_i := \id, \eta_i := \id)$ be the canonical algebra object associated to $\1$, and consider the semi-simple algebra $(S := \bigoplus_{i \in Q_0} e_i, \mu_S, \eta_S)$ formed by their direct sum; namely $\eta_S = [\eta_i]_{i\in Q_0} : \1 \ra S$ and $\mu_S: S \otimes S \ra S$ is defined on each summand $e_i \otimes e_j$ by $\mu_i$ if $i=j$ and otherwise it is zero.

To each arrow $i \xra{v(\alpha)} j$ in $Q$, we assign an object $p(\alpha) := \nPi[v(\alpha)]_{v(\alpha)-3}$.
Define $D:= \bigoplus_{\alpha \in Q_1} p(\alpha)$ and view it as a bimodule over $S$ as if we are doing ``path multiplication'', namely:
\begin{itemize}
\item as a left module over $S$, the map $\lambda: S\otimes D \ra D$ is defined on each summand $e_i \otimes p(\alpha)$ by zero if $t(\alpha) \neq i$, and otherwise it is the identity map on $p(\alpha)$ (note that $p(\alpha) = \1 \otimes p(\alpha)$).
\item the right module structure over $S$ is defined similarly, where the map $\lambda': D\otimes S \ra D$ is defined on each summand $p(\alpha) \otimes e_i$ by zero if $s(\alpha) \neq i$, and otherwise it is the identity map on $p(\alpha)$.
\end{itemize}

\begin{lemma}
We have the following isomorphism of $(S,S)$-bimodules:
\[
\overbrace{D \otimes_S D \otimes_S \cdots  \otimes_S D}^{n \text{ times}} \cong \bigoplus_{\substack{(\alpha_n, \alpha_{n-1}, ..., \alpha_1)\\ \text{ is a path} } } p(\alpha_n) \otimes p(\alpha_{n-1}) \otimes \cdots \otimes p(\alpha_1).
\]
\end{lemma}
\begin{proof}
This is a simple calculation of showing the required cokernel property, which we leave to the reader.
\end{proof}

We are now ready to define the path algebra of $Q$.
\begin{definition}\label{defn: path alg}
The \emph{path algebra} $\Path(Q)$ of $Q$ is an algebra object in $\cC(Q)$ defined as the free tensor algebra of the $(S,S)$-bimodule $D$.
To spell it out, denote 
\[
T^k(D):= \overbrace{D \otimes_S D \otimes_S \cdots  \otimes_S D}^{k \text{ times}}, 
\]
where $T^0(D) := S$ by definition.
As an object in $\cC(Q)$, the path algebra is given by
\[
\Path(Q) := \bigoplus_{k=0}^\infty T^k(D).
\]
The multiplication map $\mu: \Path(Q) \otimes \Path(Q) \ra \Path(Q)$ is induced by the tensor product over $S$, namely the map $T^k(D) \otimes T^\ell(D) \ra T^{k+\ell}(D)$ comes from the cokernel structure of $T^{k+\ell}(D) = T^k(D) \otimes_S T^\ell(D)$.
The unit map is given by the composition $\eta: \1 \xra{\eta_S} S \xhookrightarrow{\subseteq} \Path(Q)$.
\end{definition}
Note that since $Q$ is assumed to be acyclic, $D \otimes_S D \otimes_S D \otimes_S \cdots$ eventually becomes 0, hence $\Path(Q)$ is defined by a finite direct sum and in particular is an object in $\cC(Q)$.
(Otherwise the path algebra will be an algebra object in the ind completion of $\cC(Q)$ instead.)

Let us make a few observations.
Firstly, note that every $\Path(Q)$-action map $\lambda: \Path(Q) \otimes M \ra M$ on its module $M$ is completely determined by its restriction to $e_i \otimes M$ and $p(\alpha) \otimes M$ for each $i$ and each arrow $j \xra{\alpha} k$ due to the associativity requirement.
Moreover, every $\Path(Q)$-module $(M, \lambda)$ can be restricted to a $S$-module $(M, \lambda|_{S\otimes M})$.
Note that $M$ as a module over $S$ is semi-simple (since $S$ is semi-simple), decomposing into a direct sum of $e_i$-modules, which can be explicitly described as follows.
We define
\begin{equation} \label{eqn: eiM}
e_iM := \img(\lambda|_{e_i\otimes M}) \xhookrightarrow{\fri(\lambda|_{e_i\otimes M})} M
\end{equation}
as the image of the morphism $\lambda|_{e_i\otimes M}$, so that we have a unique factorisation of the morphism $\lambda|_{e_i\otimes M}$ into
\begin{equation} \label{eq: lambda epi-mono decomp}
\begin{tikzcd}
& e_iM \ar[rd, hookrightarrow, "\fri(\lambda|_{e_i\otimes M})"] & \\
M = e_i \otimes M \ar[twoheadrightarrow, ur, "\overline{\lambda|_{e_i\otimes M}}"] \ar[rr, "\lambda|_{e_i\otimes M}"] & & M .
\end{tikzcd}
\end{equation}
We have the following decomposition of $M$ (as an $S$-module and as an object in $\cC(Q)$):
\begin{lemma} \label{lem: vertex decomposition}
The morphisms 
\[
\left[\overline{\lambda|_{e_i\otimes M}} \right]_{i \in Q_0}: M \ra \bigoplus_{i \in Q_0} e_iM, \text{ and} \quad 
	\left[\fri(\lambda|_{e_i\otimes M}) \right]^T_{i \in Q_0}: \bigoplus_{i \in Q_0} e_iM \ra M
\]
are mutual inverses of $S$-modules, so that $M \cong \bigoplus_{i\in Q_0} e_iM$.
\end{lemma}
\begin{proof}
By the unital property of modules, we have that $\lambda|_{S\otimes M} \circ (\eta_S \otimes \id_M) = \id_M$.
By definition we have that $\lambda|_{S\otimes M} = [\lambda|_{e_i\otimes M}]^T_{i \in Q_0}: \bigoplus_{i \in Q_0} M = \bigoplus_{i \in Q_0} e_i \otimes M \ra M$, which together give us
\[
\id_M = 
	\lambda|_{S\otimes M} \circ (\eta_S \otimes \id_M) = 
		[\lambda|_{e_i\otimes M}]^T_{i \in Q_0}[\id_M]_{i\in Q_0} =
			\sum_{i \in Q_0} \lambda|_{e_i\otimes M}
\] 
As such, we have that 
\[
\left[\fri(\lambda|_{e_i\otimes M}) \right]^T_{i\in Q_0} \circ \left[\overline{\lambda|_{e_i\otimes M}} \right]_{i \in Q_0} = \sum_{i \in Q_0} \lambda|_{e_i\otimes M} = \id_M.
\]

On the other hand, note that the self-morphisms $\lambda|_{e_i\otimes M}$ of $M$ are mutually orthogonal idempotents.
Using the epi-mono decomposition \eqref{eq: lambda epi-mono decomp}, we get
\begin{align*}
\fri(\lambda|_{e_i\otimes M}) \circ \overline{\lambda|_{e_i\otimes M}} \circ \fri(\lambda|_{e_j\otimes M}) \circ \overline{\lambda|_{e_j\otimes M}} 
	&= \lambda|_{e_i\otimes M} \lambda|_{e_j\otimes M} \\
		&= 
		\begin{cases}
		\fri(\lambda|_{e_i\otimes M})\overline{\lambda|_{e_i\otimes M}}, &\text{ if } i = j; \\
		0, &\text{ otherwise}.
		\end{cases}
\end{align*}
In either case, the monomorphism property of $\fri(\lambda|_{e_i\otimes M})$ and the epimorphism property of $\overline{\lambda|_{e_j\otimes M}}$ allow us to conclude that 
\[
\overline{\lambda|_{e_i\otimes M}} \circ \fri(\lambda|_{e_j\otimes M}) =
	\begin{cases}
	\id_{e_iM}, &\text{ if } i = j; \\
	0, &\text{ otherwise}.
	\end{cases}
\]
Hence the composition $\left[\overline{\lambda|_{e_i\otimes M}} \right]_{i \in Q_0} \circ \left[\fri(\lambda|_{e_i\otimes M}) \right]^T_{i \in Q_0}$ is indeed the identity on $\bigoplus_{i \in Q_0} e_iM$
\end{proof}


Let us now understand the $\Path(Q)$-module structure on $\bigoplus_{i \in Q_0} e_iM$ induced by the $S$-module isomorphism above.
For each arrow $\alpha \in Q_1$, let us define $\widetilde{\lambda(\alpha)}$ to be the following composition of morphisms:
\[
\bigoplus_{i \in Q_0} (p(\alpha) \otimes e_iM) = p(\alpha)\otimes \bigoplus_{i \in Q_0} e_iM \xra{\cong} p(\alpha)\otimes M \xra{\lambda|_{p(\alpha)\otimes M}} M \xra{\cong} \bigoplus_{i \in Q_0} e_iM,
\]
which we know completely determines the induced $\Path(Q)$-action map on $\bigoplus_{i \in Q_0} e_iM$:
\[
\widetilde{\lambda}: \Path(Q) \otimes \bigoplus_{i \in Q_0} e_iM \ra \bigoplus_{i \in Q_0} e_iM.
\]
By definition, the morphism $\widetilde{\lambda(\alpha)}$ is uniquely determined by each of its component morphisms $p(\alpha) \otimes e_i M \ra e_j M$ between each pair of their summands, which is given by the composition:
\[
p(\alpha) \otimes e_i M \xhookrightarrow{\id_{p(\alpha)} \otimes \fri(\lambda|_{e_i \otimes M})} p(\alpha) \otimes M \xra{\lambda|_{p(\alpha)\otimes M}} M = e_j \otimes M \xtwoheadrightarrow{\overline{\lambda|_{e_j \otimes M}}} e_j M.
\]
We shall denote each of these component morphisms by $\widetilde{\lambda(\alpha)}_{i,j}$.
\begin{lemma} \label{lem: arrow action determined by component}
Let $\alpha$ be an arrow in $Q$.
Then $\widetilde{\lambda(\alpha)}_{i, j}$ is the zero map whenever $i \neq s(\alpha)$ or $j \neq t(\alpha)$.
In other words, the single component morphism $\widetilde{\lambda(\alpha)}_{s(\alpha),t(\alpha)}$ completely determines $\widetilde{\lambda(\alpha)}$, and hence $\widetilde{\lambda}$.
\end{lemma}
\begin{proof}
Consider the following commutative diagram
\[\begin{tikzcd}
	{p(\alpha)\otimes e_iM} && {e_j M} \\
	{p(\alpha)\otimes M} && {M = e_j \otimes M} \\
	{p(\alpha)\otimes e_i \otimes M} && {p(\alpha)\otimes M}
	\arrow["{\widetilde{\lambda(\alpha)}_{i,j}}", from=1-1, to=1-3]
	\arrow["{\id_{p(\alpha)}\otimes\mathfrak{i}(\lambda|_{e_i\otimes M})}", hook, from=1-1, to=2-1]
	\arrow["{\lambda|_{p(\alpha)\otimes M}}"', from=2-1, to=2-3]
	\arrow["{\overline{\lambda|_{e_j\otimes M}}}"', two heads, from=2-3, to=1-3]
	\arrow["{\id_{p(\alpha)}\otimes (\lambda|_{e_i\otimes M})}"', from=3-1, to=2-1]
	\arrow["{\mu|_{p(\alpha)\otimes e_i}\otimes\id_M}"', from=3-1, to=3-3]
	\arrow["{\lambda|_{p(\alpha)\otimes M}}"', from=3-3, to=2-3]
	\arrow["{\id_{p(\alpha)}\otimes(\overline{\lambda|_{e_i\otimes M}})}", shift left=3, curve={height=-30pt}, two heads, from=3-1, to=1-1]
\end{tikzcd}\]
where the bottom square is obtained from the associativity of the module action.
If $i \neq s(\alpha)$, then the bottom square is zero and by commutativity of the diagram we obtain 
\[
\widetilde{\lambda(\alpha)}_{i,j} \circ \id_{p(\alpha)}\otimes(\overline{\lambda|_{e_i\otimes M}}) = 0.
\]
Since $\id_{p(\alpha)}\otimes(\overline{\lambda|_{e_i\otimes M}})$ is an epimorphism, we get $\widetilde{\lambda(\alpha)}_{i,j} = 0$ as required.

Now consider instead the following commutative diagram
\[\begin{tikzcd}
	{p(\alpha)\otimes e_iM} && {e_j M} \\
	{p(\alpha)\otimes M} && {M = e_j \otimes M} \\
	{e_j\otimes p(\alpha)\otimes M} \\
	{p(\alpha)\otimes M} && M
	\arrow["{\widetilde{\lambda(\alpha)}_{i,j}}", from=1-1, to=1-3]
	\arrow["{\id_{p(\alpha)}\otimes\mathfrak{i}(\lambda|_{e_i\otimes M})}"', hook, from=1-1, to=2-1]
	\arrow["{\lambda|_{p(\alpha)\otimes M}}", from=2-1, to=2-3]
	\arrow["{\lambda|_{p(\alpha)\otimes M}}"', from=4-1, to=4-3]
	\arrow["{\mu|_{e_j \otimes p(\alpha)}\otimes\id_M}", from=3-1, to=4-1]
	\arrow["{=}"{marking}, draw=none, from=3-1, to=2-1]
	\arrow["{\overline{\lambda|_{e_j\otimes M}}}", two heads, from=2-3, to=1-3]
	\arrow["{\mathfrak{i}(\lambda|_{e_j\otimes M})}", shift left=4, curve={height=-30pt}, hook, from=1-3, to=4-3]
	\arrow["{\id_{e_j}\otimes\lambda|_{p(\alpha)\otimes M}}"'{pos=0.4}, shift left=1, from=3-1, to=2-3]
	\arrow["{\lambda|_{e_j\otimes M}}", from=2-3, to=4-3]
\end{tikzcd}\]
where the bottom diagram is again obtained from associativity.
If $j \neq t(\alpha)$, the bottom diagram is zero and so we get
\[
\mathfrak{i}(\lambda|_{e_j\otimes M}) \circ \widetilde{\lambda(\alpha)}_{i,j} = 0.
\]
Since $\mathfrak{i}(\lambda|_{e_j\otimes M})$ is a monomorphism, we have $\widetilde{\lambda(\alpha)}_{i,j} = 0$ as required.
\end{proof}
To conclude, we get the following:
\begin{proposition} \label{prop: arrow action determines module structure}
The $\Path(Q)$-module structure morphism 
\[
\widetilde{\lambda}: \Path(Q) \otimes \bigoplus_{i \in Q_0} e_iM \ra \bigoplus_{i \in Q_0} e_iM
\]
on  $\bigoplus_{i \in Q_0} e_iM$ induced by the $S$-module isomorphism in \cref{lem: vertex decomposition} is completely determined by the component morphisms $\widetilde{\lambda(\alpha)}_{s(\alpha),t(\alpha)}: p(\alpha) \otimes e_{s(\alpha)}M \ra e_{t(\alpha)}M$ for each arrow $\alpha \in Q_1$.
\end{proposition}

We arrive at the main result of this section.
\begin{theorem}\label{thm: rep and module equiv}
The categories $\rep(Q)$ and $\Path(Q)$-$\mod$ are equivalent as abelian, right module categories over $\cC(Q)$.
\end{theorem}
\begin{proof}
We shall construct two functors 
\[
\cF: \rep(Q) \leftrightarrows \Path(Q)\text{-}\mod : \cG
\]
and use $\cG$ to show that $\cF$ is essentially surjective and fully faithful.

First suppose we are given $V$ a representation of $Q$.
We define a $\Path(Q)$-module $M:= \cF(V)$ as follows.
As an object we have $M = \bigoplus_{i \in Q_0} V_i$, with the action morphism $\lambda: \Path(Q) \otimes M \ra M$ defined as follows:
\begin{itemize}
\item for each $i\in Q_0$, the component morphism $e_i \otimes V_i \ra V_i$ is the identity morphism;
\item for each admissible path of arrows $\alpha_{n} \circ \alpha_{n-1} \circ \cdots \circ \alpha_1$ so that 
\[
p(\alpha_n) \otimes p(\alpha_{n-1}) \otimes \cdots \otimes p(\alpha_1) \text{ is a summand of } \overbrace{D \otimes_S D \otimes_S \cdots \otimes_S D}^{n \text{ times}},
\]
the component morphism $p(\alpha_n) \otimes p(\alpha_{n-1}) \otimes \cdots \otimes p(\alpha_1) \otimes V_{s(\alpha_1)} \ra V_{t{\alpha_2}}$ is defined by the composition of morphisms $V_{\alpha_{n}} \circ (\id \otimes V_{\alpha_{n-1}}) \circ \cdots \circ (\id^{\otimes n-1} \otimes V_{\alpha_{1}})$; and
\item all other component morphisms are zero.
\end{itemize}
It is easy to check the required associative and unital properties.
Given a morphism of representations $f: V \ra W$, we define $\cF(f)$ as the direct sum of the morphisms $f_i: V_i \ra W_i$; namely the component morphisms $\cF(f)_{i,j}: V_i \ra W_j$ is the morphism $f_i$ when $i=j$ and 0 otherwise.
The fact that this defines a morphism of $\Path(Q)$-modules follows from the commutativity property $W_\alpha \circ f_i = V_{\alpha} \circ f_j$ for each arrow $i \xra{v(\alpha)} j$.

Now suppose $M$ is a module over $\Path(Q)$.
We define $V:= \cG(M)$ a representation of $Q$ as follows.
On each vertex $i \in Q_0$, set $V_i := e_iM$ with $e_iM$ as in \eqref{eqn: eiM}.
For each arrow $i \xra{v(\alpha)} j$ in $Q$, we set $V_{\alpha} = \widetilde{\lambda(\alpha)}_{i,j}$ as in \cref{lem: arrow action determined by component}.
Given a morphism of $\Path(Q)$-modules $\varphi: (M,\lambda) \ra (N, \lambda')$, we define $f := \cG(\varphi)$ as follows:
for each vertex $i \in Q_0$, $f_i$ is the following composition of morphisms:
\[
e_iM \xhookrightarrow{\fri(\lambda|_{e_i\otimes M})} M \xra{\varphi} N \xtwoheadrightarrow{\overline{\lambda'|_{e_i\otimes N}}} e_iN.
\]
The fact that this defines a morphism of representations follows from $\varphi$ being a morphism of $\Path(Q)$-modules combined with \cref{lem: arrow action determined by component}.

We now show that $\cF$ is essentially surjective and fully faithful.
By \cref{prop: arrow action determines module structure}, we have an isomorphism of $\Path(Q)$-modules
\[
\cF(\cG(M)) = \bigoplus_{i \in Q_0} e_iM \cong M
\]
for every $\Path(Q)$-module $M$, which shows that $\cF$ is essentially surjective.
It is clear from the definition that $\cF(f)=0$ implies $f=0$ for any morphism of representations $f: V \ra W$.
Since our hom spaces are vector spaces (over $\bbC$), this implies that $\cF$ is a faithful functor.
Moreover, this gives us an inequality of dimensions: 
\begin{equation} \label{eq: hom space inequality}
\dim_{\bbC}( \Hom_{\rep(Q)}(V,W)) \leq \dim_{\bbC}(\Hom_{\Path(Q)\text{-}\mod}(\cF(V), \cF(W))).
\end{equation}
As such, fullness of $\cF$ follows from showing that the inequality above is indeed an equality.
To this end, note that the functor $\cG$ also defines an injective morphism of vector spaces 
\[
\Hom_{\Path(Q)\text{-}\mod}(\cF(V), \cF(W)) \xhookrightarrow{\cG} \Hom_{\rep(Q)}(\cG(\cF(V)), \cG(\cF(W)))
\]
We leave it to the reader to check that $\cG(\cF(V)) \cong V$ and $\cG(\cF(W)) \cong W$ in $\rep(Q)$, which implies that \eqref{eq: hom space inequality} is in fact an equality, as required.

The fact that $\cF$ (and $\cG$) is exact follows directly from the definition of the abelian structure on $\rep(Q)$ and the definition of $\cF$ on morphisms.
It is also clear from the definition that both functors respect the right module category structures.
\end{proof}

\begin{remark}\label{rmk: module of fusion ring}
Note that the path algebra $\Path(Q)$ is almost never commutative (except in the trivial cases).
As such, even when $\cC(Q) \cong \TLJ_n$ for some fixed $n$, $\Path(Q)$ does not fall into the classification appearing in \cite{kirillov_ostrik_2002, ostrik_2003}.
Moreover, $\Path(Q)$ -- as with classical path algebras -- are almost never semi-simple (i.e. their category of modules are almost never semi-simple).
Nonetheless, it is easy to see that when $\cC(Q) \cong \TLJ_n$ for some fixed $n$, $K_0(\Path(Q)$-$\mod)$ falls into the classification of Etingof--Khovanov \cite[Theorem 3.4]{etinkho_95} as the direct sum of type $A$ modules over $K_0(\TLJ_n)$.
Namely, $K_0(\Path(Q)$-$\mod) \cong K_0(\rep(Q))$ is a free $K_0(\TLJ_n)$-module generated by the simple representations at each vertex.
\end{remark}

\section{Further remarks}\label{sec: further remarks}
We end this paper with a few further remarks on related works and possible extensions.
\subsection{The Coxeter groups of \texorpdfstring{$Q$ and $\check{Q}$}{Q and its unfolding}}
If $i \in Q_0$ is a sink (resp. source) in $Q$, then so is $(A,i) \in \check{Q}_0$ for each $A \in \Irr(\cC(Q))$.
In particular, it is clear that
\begin{equation}\label{eqn: reflect quiver and unfold}
\check{(\sigma_i Q)} = \prod_{A \in \Irr(\cC(Q))} \sigma_{(A,i)}\check{Q},
\end{equation}
where the order of the product does not matter since the vertices $(A,i)$ are all pairwise disjoint.

Note that $\sigma_i \mapsto \prod_{A \in \Irr(\cC(Q))} \sigma_{(A,i)}$ is also exactly the defining injective homomorphism induced by the (un)folding of the underlying Coxeter graphs (as given in \cite[Corollary 3.3]{Lusz83}  or equivalently \cite[Proposition 2.3]{Crisp99}).
This relation can be lifted to the functorial level on representations of quivers as follows.
With $S^\pm_{(A,i)}$  denoting the usual reflection functors on sink (resp. source) of classical quivers and  $\cU$ denote the equivalence defined in the proof of \cref{thm: rep Q and checkQ equiv}, the following diagram is commutative (up to natural isomorphisms):
\[
\begin{tikzcd}[column sep = 4cm]
	{\rep(Q)} & {\rep(\sigma_iQ)} \\
	{\rep(\check{Q})} & {\rep(\prod_{A\in\Irr(\cC(Q))}\sigma_{(A,i)}\check{Q}).}
	\arrow["{R^\pm_i}", from=1-1, to=1-2]
	\arrow["\cU"', from=1-2, to=2-2]
	\arrow["\cU"', from=1-1, to=2-1]
	\arrow["{\prod_{A\in\Irr(\cC(Q))} S^\pm_{(A,i)}}", from=2-1, to=2-2]
\end{tikzcd}
\]
In other words, as in the classical theory, the representations of finite-type Coxeter quivers can be viewed as a (categorical) realisation of Coxeter root systems over commutative rings in the sense of  \cite{Dyer09}, providing categorical interpretation of some of the results in loc. cit.
\begin{remark} 
The commutative diagram above also implies that $R^-_i$ is a left adjoint of $R^+_i$.
\end{remark}

\subsection{Representations in other fusion categories and module categories}
It is intriguing to find out what other fusion categories can be considered as ``natural candidates'' for a nice theory of generalised quiver representations.
Even if we focus on Coxeter quivers, certain choices can be changed without affecting the end results.
For example, in the odd $n$ cases, one could've used $\cC_n = \TLJ_n$ instead of just the even part; the reason that we took the even part is so that when $n=3$ we get the category (equivalent to) $\vec_{\bbC}$.
Moreover, instead of tensoring with $\nPi_{n-3}$ for the arrows labelled by $n$, we could instead tensor by $\nPi_1$.

On labels $n=12$ and $n=30$, we could also use the type $E_6$ and type $E_8$ fusion categories in place of $\TLJ_{12}$ (type $A_{11}$) and $\TLJ_{30}$ (type $A_{29}$) respectively, where $\nPi_{n-3}$ is replaced by the generating object of the respective category (see \cite[Section 2.3]{Edie-Michell20} and references therein for a description of the $E_6$ and $E_8$ fusion categories).
In these cases, the category of representations of the $I_2(12)$ and $I_2(30)$ Coxeter quivers will be equivalent to the category of representations of the classical quiver whose underlying graph is of type $E_6$ and type $E_8$ respectively.

One could also try and work over positive characteristic using the Verlinde category introduced by Ostrik in \cite{ostrik_2020} (note that we never used the braiding structure in the fusion categories that we have considered).

Just as quivers are useful in studying finite-dimensional algebras, \cref{thm: rep and module equiv} also suggests that it could be worth exploring the usage of quivers in understanding algebras in fusion categories, which could also lead to the understanding of (finite) module categories over fusion categories.

\subsection{Relation to cluster algebras and their categorifications} \label{subsec: cluster and folding}
In the process of working on this paper, the author came across the paper of Duffield--Tumarkin \cite{DufTumarkin22}, which approaches a related problem directly from the representations of classical quivers.
Their motivation is different from ours; the aim was to categorify certain cluster algebra combinatorics and they focus purely on the folding of finite type quivers $\check{Q}$ to obtain induced semi-ring actions on $\rep(\check{Q})$.
The semi-ring they consider agrees with the positive part of our fusion ring $K_0(\cC(Q))$.
Moreover, their semi-ring action on $\rep(\check{Q})$ can indeed be lifted to an action of $\cC(Q)$, which coincides with our induced action of $\cC(Q)$ on $\rep(\check{Q})$ through the equivalence $\rep(Q) \cong \rep(\check{Q})$ in \cref{thm: rep Q and checkQ equiv}.
Compare their Theorem 1.1 and Theorem 1.2 with our \cref{thm: gabriel root} and \cref{thm: rep Q and checkQ equiv} respectively.

Due to the incompetence of the author in the theory of cluster algebras, we shall not comment further on the (possible) relevance of our results in this paper with cluster combinatorics, where the interested reader are referred to \cite{DufTumarkin22} for the finite type cases.

\printbibliography

\end{document}